\documentclass[12pt]{article}    
\usepackage{amsmath, amsfonts, amssymb}

\usepackage{amsthm}

\newtheorem{defn}{Definition}[section]
\newtheorem{theo}[defn]{Theorem}
\newtheorem{lemma}[defn]{Lemma}
\newtheorem{prop}[defn]{Proposition}
\newtheorem{coro}[defn]{Corollary}
\theoremstyle{remark}
\newtheorem{remark}{Remark}[section]

\def\a{\alpha}
\def\b{\beta}

\def\s{\sigma}
\def\t{\tau}

\def\conmod#1#2{\equiv #1\ \hbox{$($mod\ }#2\hbox{$)$}}

\newcommand{\algclosure}{\overline{\mathbb{F}}_p}

\newcommand{\co}{{\mathcal O}}

\newcommand{\cp}{{\mathcal P}}

\newcommand{\gp}{\mathbb{P}}

\title{On the Classification of Exceptional 
Planar Functions over $\mathbb{F}_{p}$ }

\author{Fernando Hernando\footnote{Research of the first author supported by  Spain Ministry of Education
MTM2007-64704 and Bancaixa P1-1B2009-03} \\
Department of Mathematics\\
Universidad Jaume I\\
Spain \\
Gary McGuire\footnote{Research of the second author supported by the Claude
Shannon Institute, Science
Foundation Ireland Grant 06/MI/006.}\\
School of Mathematical Sciences\\
University College Dublin\\
Ireland\\
Francisco Monserrat$^*$ \\
Instituto Universitario de Matem\'{a}tica Pura y Aplicada\\
Universidad Polit\'{e}cnica de Valencia\\
Spain \\}

\begin{document}

\maketitle


\abstract{
We will present many strong partial  results towards a classification 
of exceptional planar/PN monomial functions on finite fields.
The  techniques  we use are
the Weil bound, Bezout's theorem, and Bertini's theorem.
 }

\bigskip



\section{Introduction}

We present  some new results on the classification 
of perfect nonlinear (PN) or planar  functions.  
These have connections to finite geometry,
coding theory and cryptography.

Let $p$ be a prime number, $t\geq 3$ an integer, and let $f_t(x,y)$ the polynomial 
\[
f_t(x,y):=(x+1)^t-x^t-(y+1)^t+y^t\in \mathbb{F}_p[x,y].
\]
Notice that $x-y$ divides $f_t(x,y)$.
We define another polynomial $g_t(x,y)$ by
\[
g_t(x,y):=\frac{(x+1)^t-x^t-(y+1)^t+y^t}{x-y} \in \mathbb{F}_p[x,y].
\]

In this paper we consider the following conjecture (see Section 2 for the background).

 
\bigskip

\noindent\textbf{Conjecture PN3:} {\em 
Suppose $t>2$.  The polynomial $g_t(x,y)$ has an absolutely irreducible factor defined over $\mathbb{F}_p$ for all $t$ not of the form $p^i+1$ (when $p\geq 3$) and 
$(3^i+1)/2$ (when $p=3$).}
\bigskip

Let $t=p^i\ell+r$, where $r$ is the remainder upon division of $t$ by $p$, and $i=\max \{ j :  p^j \mbox{ divides } t-r\}$. 
We will see later that we may assume that
$t$ is relatively prime to $p$, i.e., we may assume $r\neq 0$.

The results in this paper are the following.

\begin{theo}\label{maintheo}
The polynomial $g_t(x,y)$ has an absolutely irreducible factor defined over $\mathbb{F}_p$ in the following cases.
\begin{itemize}
\item[(A)] $t\not \equiv 1 \ mod \ p$.
\begin{itemize}
\item[(a)] Either $\gcd(p-1,t)\geq 3$ or $\gcd(p-1,t-1)\geq 2$.
\item[($\hat{a}$)] If there exists $m\in\mathbb{N}$ such that $g_t(x,y)$ does not factor over $\mathbb{F}_{p^m}$ and, furthermore, either $\gcd(p^m-1,t)\geq 3$ or $\gcd(p^m-1,t-1)\geq 2$.
\item[(b)] The number of singular points of the curve defined by $g_t(x,y)=0$ (over an algebraic closure of $\mathbb{F}_p$) is less than $\frac{(t-2)^2}{4}$.
\item[(c)] $t$ is even, $(s-1)^{t-1}\not\in \mathbb{F}_{p}$ and $(s-1)^{(t-1)(p-1)}\not=-1$ for all $s\not=1$ in the set of $(t-1)$-roots of unity in al algebraic closure of $\mathbb{F}_p$.

\item[(d)] $t$ is even and $t-1$ divides $p^{2e}+1$ for some positive integer $e$.

\item[(e)] $t-1\geq 3$ is a prime number such that the multiplicative order of $p$ in  $\mathbb{Z}/(t-1)\mathbb{Z}$ is $(t-2)/2$.

\item[(f)] If $e:=\gcd(e_{t-1},e_t)$, $e_n$ denoting the multiplicative order of $p$ in $\mathbb{Z}/n\mathbb{Z}$,
\begin{itemize}
\item[(1)] either there exists a divisor $d>2$ of $t$ such that $\gcd(e,e_d)=1$ or
\item[(2)] there exists a divisor $d>1$ of $t-1$ such that $\gcd(e,e_d)=1$.
\end{itemize}
\item[(g)] The polynomial $g_t(x,y)$ is irreducible over $\mathbb{F}_p$ and $gcd(e_d\mid d\in E)=1$, where $e_d$ is defined as in (f) and $E:=\{d\in \mathbb{N}\mid d>2 \mbox{ and} \mbox{ $d$ divides either $t$ or $t-1$}\}$.

\end{itemize}

\item[(B)]  $t\equiv 1 \ mod \ p$. 
\begin{enumerate}
\item[(B.1)] $gcd(\ell,p^i-1)< \ell$ and at least one of these conditions holds:
\begin{itemize}
\item $p\geq 5$, $i\geq 1$ and $\ell> 3$, 
\item $p\geq 5$, $i\geq 2$ and $\ell\geq 3$,  
\item $p=3$, $i\geq 2$ and $\ell\geq 3$,
\end{itemize}
\item[(B.2)]  $gcd(\ell,p^i-1)=\ell$ and $\ell<p^i-1$.
\end{enumerate}

\end{itemize}

\end{theo}

The conditions of Case (A) in the above theorem satisfy the following implications:
$$(d)\; \vee\; (e)\; \Rightarrow \; (c)\; \Rightarrow (b).$$
 Conditions $(b)$ and $(c)$ concern an algebraic closure of the field $\mathbb{F}_p$. Conditions $(d)$ and $(e)$ are weaker but they involve only integer arithmetic.

The layout of this paper is as follows.
In Section 2 we give the background to the problem and explain the motivation
and origin of Conjecture PN3.
Section 3 gives a more detailed background of the results and
techniques we use.
In Section \ref{anasing} we analyze the singular points of the curve defined
by $g_t(x,y)=0$, and in Section \ref{anaintersect} we calculate or estimate
the intersection multiplicities of hypothetical factors at the singular points.
Section \ref{part2proof} proves the Case $(B)$ results under the assumption
that the curve is irreducible over $\mathbb{F}_p$.
Section  \ref{part3proof} removes this assumption using more careful analysis.
Finally Section \ref{caseAresults} proves our results on Case $(A)$.

\section{Background and Motivation}

In this section we present background to the problems under
consideration in this paper. 

\subsection{PN and Planar Functions}
Let $p$ be a prime number and let $q=p^n$.
Recall that any function $\mathbb{F}_q \longrightarrow \mathbb{F}_q$
can be expressed uniquely as a polynomial function (with 
coefficients in $\mathbb{F}_q$) of degree less than $q$.
A polynomial function is called a permutation polynomial (PP) if
it is a bijective function $\mathbb{F}_q \longrightarrow \mathbb{F}_q$.

\begin{defn}
A function $f : \mathbb{F}_q \longrightarrow \mathbb{F}_q$
is said to be \emph{planar} if the functions
$f(x+a)-f(x)$ are PPs for all nonzero $a\in \mathbb{F}_q$.
\end{defn}

Planar functions are used to construct finite projective planes,
and have been studied by finite geometers since at least 1968
(Dembowski and Ostrom \cite{DO}).
Note that planar functions cannot exist in characteristic 2,
because if $D_a(x):=f(x+a)-f(x)$ and $D_a(x)=b$ then
$D_a(x+a)=b$ also.

\begin{defn}
A function
$f:\mathbb{F}_{q}\longrightarrow \mathbb{F}_{q}$ is said to be PN
(Perfect Nonlinear) if for every $a,b\in\mathbb{F}_{q}$
with $a\neq 0$ we have
$$
\sharp\{x\in\mathbb{F}_{q}\mid f(x+a)-f(x)=b\}\leq 1.
$$
\end{defn}

PN functions were first defined in 1992 by Nyberg and Knudsen \cite{NK},
in a cryptography paper. 
Note that PN functions cannot exist in characteristic 2,
because if $x$ is a solution to  $f(x+a)-f(x)=b$ then $x+a$ is another solution.

It is clear that PN functions and planar functions are the same thing!
They have different origins; PN functions come from cryptography whereas
planar functions come from finite geometry.

We consider monomial functions in this article. 
Because $x^t$ is planar iff $x^{pt}$ is planar, we may assume that
$t$ is relatively prime to $p$.

The known planar monomials $f(x)=x^t$ are in the following table.

\begin{table}[!h]
\noindent\begin{center} 
\begin{tabular}{|c|c|c|c|c|} 
\hline 
  Characteristic& \footnotesize{Exponents $t$} & \footnotesize{Conditions} & \footnotesize{Proved by}\\ 
\hline 
\hline 
\footnotesize{odd} & \footnotesize{$2$} & \footnotesize{None} & \footnotesize{Classical}\\ 
  
\hline 
\footnotesize{odd}  & \footnotesize{$p^i+1$ }&  \footnotesize{$n/(i,n)$ odd}  &\footnotesize{Dembowski-Ostrom}\\ 
 
\hline 
\footnotesize{3} &\footnotesize{$(3^i+1)/2$}& \footnotesize{$(i,n)=1$, $i$ odd}& \footnotesize{Coulter-Matthews}\\ 
\hline 
\end{tabular} 
\end{center}
\caption{Known PN exponents  $t$ }\label{tablePN}
\end{table}

It is conjectured that this list is complete:

\bigskip
\textbf{Conjecture PN1:} {\em All planar functions of the form  $x^t$ 
are listed in Table 1.}

\vspace{.2cm}

In this article we present some partial results towards this conjecture.
We consider the classification of functions $x^t$ that are
planar/PN on $\mathbb{F}_{p^n}$ for infinitely many $n$.
The known examples in Table \ref{tablePN} all have this property.
Therefore, a weaker conjecture than Conjecture 1 is the following:

 \bigskip
\textbf{Conjecture PN2:} {\em If $x^t$ is a planar function 
on $\mathbb{F}_{p^n}$ for infinitely many $n$,
then $t$ is 
of the values listed in the table.}

\vspace{.2cm}

For monomial functions $f(x)=x^t$,  
it was shown in \cite{Coulter-Matthews} that 
$x^t$ is planar over $\mathbb{F}_q$ if and only if
$(x+1)^t - x^t$ is a PP over $\mathbb{F}_q$,
i.e., for monomial functions we only need consider the $a=1$ case of Definition 1.

\begin{defn}
A PP $f(x)\in \mathbb{F}_q [x]$ is called \emph{exceptional} if 
$f$ is a PP on infinitely many extension fields of $\mathbb{F}_q$.
\end{defn}

Exceptional PPs have been the subject of many papers, see \cite{GRZ} for example.
Their monodromy groups are of great interest and have been classified.

To prove Conjecture PN2,
we consider the function $f(x)=x^t$ on the base field $\mathbb{F}_p$,
and we would like to  prove that
$(x+1)^t - x^t$ is not an exceptional PP on $\mathbb{F}_p$
when $t$ is not one of the values listed.

Observe  that $(x+1)^t - x^t$ is not a PP over $\mathbb{F}_{p^n}$ 
if there exist $\mathbb{F}_{p^n}$-rational points $(x,y)$ on the curve
defined by
$$
f_t(x,y)=(x+1)^t-x^t-(y+1)^t+y^t
$$
with $x\not= y$.
It is obvious that $f_t(x,y)$ has $x-y$ as a factor. 
Therefore, we would like to know whether the curve defined by
$$
g_t(x,y)=\frac{(x+1)^t-x^t-(y+1)^t+y^t}{x-y}
$$
has rational points over $\mathbb{F}_{p^n}$ with $x\not= y$.
Note that $g_t(x,y)$ is defined over $\mathbb{F}_{p}$.

The following is easily proved using the Weil bound.

\begin{theo}\label{absirredpropPN}
If $g_t(x,y)$ has an absolutely irreducible factor defined over  $\mathbb{F}_p$
then $g_t(x,y)$ has
rational points  $(\alpha,\beta)\in (\mathbb{F}_{p^n})^2$ with
distinct coordinates for all $n$ sufficiently large.
\end{theo}

Based on the known examples of PN functions in Table 1, 
we make the following conjecture.

\bigskip

\noindent\textbf{Conjecture PN3:} {\em 
Suppose $t>2$.  The polynomial $g_t(x,y)$ has an absolutely irreducible factor defined over $\mathbb{F}_p$ for all $t$ not of the form $p^i+1$ (when $p\geq 3$) and 
$(3^i+1)/2$ (when $p=3$).}
\bigskip

Clearly 
Conjecture PN3 implies Conjecture PN2.
Therefore, the topic of this paper is proving Conjecture PN3.
We give some partial results, which are stated in the introduction; the full conjecture is still open.

\subsection{Small values of $p$ and $t$}

To demonstrate the power of our results, 
we have implemented MAGMA functions for testing the conditions given in Theorem \ref{maintheo} and, using them, we have proved Conjecture PN3 for 
a great many values $t\leq 1000$ when $p$ is either $3$, $5$ or $7$.
Condition $(\hat{a})$ has been checked only for $m=2$ and $m=3$ and, to implement Condition $(d)$, we have used the equivalent formulation given in Corollary \ref{newprop}. 
Condition $(b)$ is quite strong but it involves  computations with a variety which is computationally intensive and the MAGMA function that implements it does not finish; so we have not taken it into account for our tests.


Notice that the conditions concerning Case $(B)$ involve only integer arithmetic. 
However, in Case $(A)$, conditions $(a)$, $(d)$, $(e)$ and $(f)$ are purely arithmetical conditions relating $t$ and $p$ (integer arithmetic) and the other conditions involve computations concerning elements in the algebraic closure of $\mathbb{F}_p$. Taking this observation into account, we consider separately the following groups of conditions associated with Case $(A)$:

\begin{itemize}
\item Group 1: $(a)$, $(d)$, $(e)$ and $(f)$.
\item Group 2: Conditions of Group 1 and $(g)$.
\item Group 3: Conditions of Group 2 and $(c)$.
\item Group 4: Conditions of Group 3 and $(\hat{a})$ (taking $m=2$).
\item Group 5: Conditions of Group 3 and $(\hat{a})$ (taking $m=3$).
\end{itemize}

Note that each Group includes the previous Group (except 4 and 5).

Tables 2, 3 and 4 show, for $p=3,5,7$, the values of $t$, $3\leq t\leq 1000$, for which 
we are \emph{not} able to prove Conjecture PN3 using the conditions in each group. 
Notice that for $p\in \{5,7\}$ conditions in 
Group 5 cover every $t$ from $3$ to $1000$ corresponding to Case $(A)$.

Of course we omit the exceptional values of $t$ when doing these computations.

Putting all our results together, these computations show the following.

\begin{theo}
Conjecture PN3 is proved by Theorem \ref{maintheo} in the following cases.
\begin{enumerate}
\item $p=3$, all values of $t<1000$ except 758 and $t \conmod{4}{6}$
\item $p=5$, all values of $t<1000$ except $15$ and $76$
\item $p=7$, all values of $t<1000$ except $22$ and $148$
\end{enumerate}
\end{theo}

\begin{table}[!h]
\noindent\begin{center} 
\begin{tabular}{|c|c|} 
\hline 
Conditions &  Excluded values of $t$ for $p=3$\\ \hline 
\hline 
$(B)$ &  118 values of the form $4+6k$  \\
\hline 
$(A)$: Group 1 &  79 values \\ 
\hline 
$(A)$: Group 2 &  69 values\\ 
\hline 
$(A)$: Group 3 &    482, 758\\ 
\hline 
$(A)$: Group 4 &    482, 758\\
\hline 
$(A)$: Group 5 &    758\\
\hline 
\end{tabular} 
\end{center}
\caption{Not included values in Theorem 1.2 for $p=3$}
\end{table}

\begin{table}[!h]
\noindent\begin{center} 
\begin{tabular}{|c|c|} 
\hline 
Conditions  in Th 1.2 &  Excluded values of $t$ for $p=5$\\ \hline 
\hline 
$(B)$ &  16,  76  \\
\hline 
$(A)$: Group 1 &  34 values\\
\hline
$(A)$: Group 2 &  24 values\\
\hline 
$(A)$: Group 3 & 82, 218, 274, 322, 334, 442, 
 658,  898\\
\hline 
$(A)$: Group 4 &  218\\
\hline 
$(A)$: Group 5 &  \\
\hline 
\end{tabular} 
\end{center}
\caption{Not included values in Theorem 1.2 for $p=5$}
\end{table}

\begin{table}[!h]
\noindent\begin{center} 
\begin{tabular}{|c|c|c|} 
\hline 
Conditions  in Th 1.2 &  Excluded values for $p=7$ \\ 
 \hline 
\hline 
$(B)$  &  22,  148  \\ 
\hline 
$(A)$: Group 1 &  54 values\\ 
\hline 
$(A)$: Group 2 &  41 values\\ 
\hline 
$(A)$: Group 3 &  362, 818\\ 
\hline 
$(A)$: Group 4  & 362, 818\\
\hline 
$(A)$: Group 5  & \\
\hline
\end{tabular} 
\end{center}
\caption{Not included values in Theorem 1.2 for $p=7$ }
\end{table}

We note that similar results involving Case $(B)$ have also been achieved independently by Robert Coulter \cite{C}. Moreover a preprint has been posted on the arxiv
by Elodie Leducq \cite{EL} solving Case $(B)$.

\section{Detailed Background Results}

We wish to prove Conjecture PN3.
The idea is to show that Bezout's theorem cannot possibly hold, when applied to two
(or more) putative factors of the polynomial $g_t$.
This proof depends heavily on analyzing the singular points of the curve $g_t$.
  
\subsection{Background on curves}

Let $\overline{\mathbb{F}}_p$ be an algebraic closure of $\mathbb{F}_p$.
A polynomial $h(x,y)\in \overline{\mathbb{F}}_p[x,y]$ defines an affine plane curve
$$C_h:=\{(\a,\b)\in \overline{\mathbb{F}}_p^2\mid h(\a, \b)=0\}.$$
Given a point $P=(\a,\b)\in \algclosure^2$ we write
$$h(x+\a,y+\b)=H_0(x,y)+H_1(x,y)+H_2(x,y)+\cdots,$$
where each $H_i(x,y)$ is either 0 or a homogeneous polynomial of degree $i$.  
\begin{defn}
The multiplicity of $h$ at $P$ is the smallest $m$ with $H_m\not=0$,
and is denoted by $m_P(h)$ or $m_P(C_h)$.
\end{defn}
In particular, $P\in C_h$ if and only if $m_P(h)\ge 1$.

We say  that $P$ is a {\it singular point of $h$} (or {\it of $C_h$}) if  $m_P(h)\ge 2$.
The linear factors of $H_m$ are the tangent lines 
to the curve $C_h$ at the point $P$.
The collection of tangent lines is called the {\it tangent cone}.

We consider the projective plane $\mathbb{P}^2$ over $\overline{\mathbb{F}}_p$ and take homogeneous coordinates $(X:Y:Z)$ such that $x:=X/Z$ and $y:=Y/Z$ are affine coordinates in the chart defined by $Z\not=0$. Let $F_t(X,Y,Z)$ (resp., $G_t(X,Y,Z)$) be the homogenization of the polynomial $f_t(x,y)$ (resp., $g_t(x,y)$) and denote by 
 $\chi_t$ the projective curve over $\overline{\mathbb{F}}_p$ defined by the equation $G_t(X,Y,Z)=0$. Notice that $g_t(x,y)$ has an absolutely irreducible factor over $\mathbb{F}_p$ if and only if $G_t(X,Y,Z)$ does so.

\subsection{Background on intersection multiplicity}

Bezout's theorem  is a classical result in algebraic geometry and appears frequently in the literature (see for example chapter 5 of \cite{F}).

\bigskip

{\textbf{Bezout's Theorem}:} Let $r$ and $s$ be two
projective plane curves of degrees $D_1$ and $D_2$ over an algebraically closed 
field $k$  having no components in common. Then,
\begin{equation}\label{Bezout}
\sum_{P}I(P,r,s)=D_1D_2.
\end{equation}
The sum runs over all points $P$ in the projective plane $\mathbb{P}^2(k)$, 
and by $I(P,r,s)$ we mean
the intersection multiplicity of the curves $r$ and $s$ at the
point $P$. Notice that if $r$ or
$s$ does not go through $P$, then $I(P,r,s)=0$. Therefore, the sum
in (\ref{Bezout}) runs over the singular points of the product $rs$.
In our case, the sum will run over the singular points of $g_t$.

Using properties $I(P,r_1r_2,s)=I(P,r_1,s)+I(P,r_2,s)$ and
$\deg (r_1r_2)=\deg (r_1) +\deg (r_2)$ one can generalize
Bezout's Theorem to several curves $f_1$, $f_2$, $\cdots$ ,$f_r$
as follows:
\begin{equation}\label{bezoutseveral}
\sum_P \sum_{1\leq i<j\leq r}  I(P,f_j,f_j)= \sum_{1\leq i<j\leq r }\deg (f_j)\deg (f_j).
\end{equation}

We refer the reader to chapter 5 of \cite{F} for the definition of 
the intersection multiplicity $I(P,r,s)$ of two curves $r,s$ at a point $P$.
The following property of the intersection multiplicity will be useful for us.
It is part of the definition of intersection multiplicity in \cite{F}.
We state it as a Corollary.

\begin{coro}\label{DifferentTangentCone}
\begin{equation}\label{IntersectionMultiplicity}
I(P,r,s)\geq m_P(r)m_P(s),
\end{equation}
and  equality holds if and only if the tangent cones of $r$ and $s$ do not share any linear factor.
\end{coro}

We note that the degree of $g_t$ is $t-1$.
Therefore, if $g_t=uv$ then our strategy is to show that 
$\sum_{P}I(P,u,v) < (\deg u)(\deg v)$
by analyzing the singular points $P$.
We usually lower bound the product of the degrees, and upper
bound the sum of intersection multiplicities, and show that
the upper bound is strictly less than the lower bound,
to obtain our contradiction.

\subsection{Background on  pencils and clusters of base points}\label{backbertini}

Let $\gp^2$ be the projective plane over $\algclosure$. Given a positive integer $d$, the projectivization of a vector subspace $\cp$ of $H^0(\gp^2,\co_{\gp^2}(d))$ with (projective) dimension 1 is called a \emph{pencil} on $\gp^2$, where $\co_{\gp^2}$ denotes the structure sheaf \cite[II.2]{Hartshorne} of $\gp^2$ as an algebraic variety over $\overline{\mathbb{F}}_p$.

Let us fix projective coordinates $(X:Y:Z)$ on $\gp^2$. Using them, such a pencil can be seen as the space of projective curves with equations $\alpha F(X,Y,Z)+\beta G(X,Y,Z)=0$, where $F(X,Y,Z)$ and $G(X,Y,Z)$ are fixed homogeneous polynomials of $\algclosure[X,Y,Z]$ of degree $d$ and $(\alpha:\beta)$ varies along $\gp^1$, the projective line over $\algclosure$. The pencil generated by polynomials $F$ and $G$ will be denoted by $\cp(F,G)$. Also, from now on, we shall assume that all considered pencils have no fixed components, that is, there is no curve that is component of all the curves of the pencil. A \emph{base point} of a pencil $\cp$ is a (closed) point $P\in \gp^2$ that belongs to all the curves of $\cp$.

Let us consider any sequence of morphisms
\begin{equation}
\label{seq} X_{n+1} \mathop  {\longrightarrow} \limits^{\pi _{n} }
X_{n} \mathop {\longrightarrow} \limits^{\pi _{n-1} }  \cdots
\mathop {\longrightarrow} \limits^{\pi _2 } X_2 \mathop
{\longrightarrow} \limits^{\pi _1 } X_1 : = \gp^2,
\end{equation}
where $\pi_i$ is the blow-up of $X_i$ at a closed point $p_i\in
X_i$, $1\leq i\leq n$. The associated set of closed points ${\mathcal C} =
\{p_1,p_2,\ldots,p_n\}$ will be called a {\it cluster} (of infinitely near points) over
$\gp^2$. For each point $p \in {\mathcal C}$, set $\tilde{E}_p$ (resp.,
$E^*_p$) the strict (resp., total) transform on $X_{n+1}$ of the
exceptional divisor created by the blowing-up at $p$. For any curve $C$ on $\gp^2$ it holds that
\begin{equation}\label{a}
C^*=\tilde{C}+\sum_{p\in BP(\cp)} m_p(C) E_p^*,
\end{equation}
where $C^*$ (resp., $\tilde{C}$) denotes the total (resp., strict)
transform of $C$ on $Z_{\cp}$ and $m_p(C)$ is the multiplicity
at $p$ of the strict transform of $C$ on the surface to which $p$
belongs. See, for instance, \cite[IV.2]{Eisenbud}, for the definition and properties of the blow-up, strict transforms and total transforms.

Given a pencil $\cp(F,G)$, the quotient $F/G$ gives rise to a rational map $f_{\cp}: X
\cdots \rightarrow \gp^1$ that is independent from the chosen basis $\{F,G\}$
up to composition with an automorphism of $\gp^1$. The closures
of the fibers of $f_{\cp}$ are exactly the curves of the pencil
$\cp$. Moreover, there exists a minimal composition
of blow-ups $\pi_{\cp}: Z_{\cp} \longrightarrow X$ (as in (\ref{seq}))
eliminating the indeterminacies of
the rational map $f_{\cp}$, that is,
the map $h_{\cp}:= f_{\cp} \circ
\pi_{\cp}:Z_{\cp}\rightarrow \gp^1$ is a morphism \cite[II.6]{Beauville}. 
The set of centers of the blow-ups giving rise to $\pi_{\cp}$, that we denote by $BP(\cp)$, is called the {\it cluster of base points} of
$\cp$.


Set $(x:=\frac{X}{Z}, y:=\frac{Y}{Z})$ affine coordinates in the affine chart defined by $Z\not=0$. Let us consider two homogeneous polynomials of the same degree $P(X,Y,Z),Q(X,Y,Z)\in \overline{\mathbb F}_p[X,Y,Z]$ (and without common components) and the associated pencil $\cp={\mathcal P}(P,Q)$. It is said that this pencil is \emph{composite} if there exists a rational function of $\gp^1$, $r=\frac{R_1(X,Y)}{R_2(X,Y)}$ ($R_1$ and $R_2$ being homogeneous polynomials of the same degree $\geq 2$) and a rational function $g$ of $\gp^2$ such that $f_{\cp}=r\circ g$. Also, ${\mathcal P}(P,Q)$ is \emph{irreducible} if all but finitely many curves of ${\mathcal P}(P,Q)$ are (absolutely) irreducible curves. A classical theorem of Bertini (valid in characteristic 0) that characterizes reducible linear systems was generalized for positive characteristic (see, for instance, either \cite{kleiman} and references therein, or \cite[Th. 7.19]{Iitaka}, or \cite[Th. 2.2]{Bodin}) giving rise, in the particular case of pencils (which are linear systems of dimension 1), to the following result:

\begin{theo}\label{bertini}

Let $P(X,Y,Z)$ and $Q(X,Y,Z)$ as above. If the pencil ${\mathcal P}(P,Q)$ is not composite, then ${\mathcal P}(P,Q)$ is irreducible.

\end{theo}

The following corollary will be used in the proof of Lemma \ref{lemagordo}.

\begin{coro}\label{bbb}

Let $P(X,Y,Z)$ and $Q(X,Y,Z)$ as above and let $p$ be a base point of the pencil $\cp={\mathcal P}(P,Q)$ such that the multiplicities at $p$ of all but finitely many of the curves in ${\mathcal P}$ are equal to 1. Then ${\mathcal P}$ is irreducible.
\end{coro}

\begin{proof}
We shall reason by contradiction. Therefore, assume that ${\mathcal P}(P,Q)$ is not irreducible. Then, by Theorem \ref{bertini}, ${\mathcal P}(P,Q)$ is a composite pencil. So, there exists a rational function $r$ of $\gp^1$ of degree $\geq 2$ and a rational function $g=\frac{P'(X,Y,Z)}{Q'(X,Y,Z)}$ of $\gp^2$ ($P'$ and $Q'$ are homogeneous polynomials of the same degree) such that $f_{\cp}=r\circ g$. This implies that, on the one hand, $p$ is a base point of the pencil ${\mathcal P}(P',Q')$ and, on the other hand, each element of ${\mathcal P}(P,Q)$ is a product of elements of ${\mathcal P}(P',Q')$. These two facts lead to a contradiction because the multiplicity at $p$ of a general element of ${\mathcal P}(P,Q)$ is 1.

\end{proof}

Consider a cluster of infinitely near points ${\mathcal C}$ over $\gp^2$ and a map ${\bold m}:{\mathcal C}\rightarrow \mathbb{N}$. For each positive integer $d$ we shall denote by ${\mathcal L}_d({\mathcal C},{\bold m})$ the projectivisation of the vector space over $\overline{\mathbb{F}}_p$ of all homogeneous polynomials in $\overline{\mathbb{F}}_p[X,Y,Z]$ of degree $d$ defining curves $C$ of $\gp^2$ such that the divisor $C^*-\sum_{p\in S} {\bold m}(p)E_p^*$ (on the surface obtained after blowing-up the points in ${\mathcal C}$)  is effective. The following lemma is proved in \cite[Prop. 3.4]{cgm} in a more general framework.

\begin{lemma}\label{ccc}
Assume that ${\mathcal P}(P,Q)$ is irreducible and set $d:=\deg(P)=\deg(Q)$. Let ${\bold m}: BP(\cp)\rightarrow \mathbb{N}$ be the map that assigns, to each $p\in BP(\cp)$, the multiplicity at $p$ of a general element of ${\mathcal P}(P,Q)$. Then ${\mathcal L}_d(BP(\cp),{\bold m})={\mathcal P}(P,Q)$.

\end{lemma}

\begin{lemma}\label{eee}
Let $K$ be a subfield of $\overline{\mathbb{F}}_p$ and consider a cluster $\mathcal C$  over $\gp^2$ such that every $p\in {\mathcal C}$ is a $K$-rational point of the surface to which it belongs. Let $d$ be a positive integer and let ${\bold m}:{\mathcal C}\rightarrow \mathbb{N}$ be a map. Then the projective space (over $\overline{\mathbb{F}}_p$) ${\mathcal L}_d({\mathcal C},{\bold m})$ has a basis whose elements are polynomials with coefficients in $K$.
\end{lemma}

\begin{proof}

Consider an homogeneous polynomial of degree $d$ with indeterminate coefficients. If one imposes the conditions defining ${\mathcal L}_d({\mathcal C},{\bold m})$ to this ``generic'' polynomial, it is obtained a system of linear equations with coefficients in $K$ whose solution set is ${\mathcal L}_d({\mathcal C},{\bold m})$. The result follows trivially from this observation.

\end{proof}

\section{Analysis of Singularities}\label{anasing}

Before we begin the proofs of our results towards Conjecture PN3,
we need to study the singular points of the curve $\chi_t$.
The proof follows the same lines as the proof in \cite{HM1}, \cite{HM2}.
We first analyze the singular points, we find a description of them, 
count their number and multiplicity.

Let $r$ be the residue of $t$ modulo $p$, so we may write
\[
t=p^i\ell+r \quad \text{with} \quad 0\leq r< p
\]
and where $\ell$ is not divisible by $p$. We similarly also write
\[
\ell=ps+j \quad \text{with} \quad 0< j< p
\]  
where this time $s$ could be divisible by $p$.

Because $x^t$ is planar iff $x^{pt}$ is planar, we may assume that
$t$ is relatively prime to $p$, i.e., we may assume $r\neq 0$.

We proved  Conjecture PN3 in the case that $t\equiv 1$ mod $p$ and $\ell \equiv$ 1 mod $p$
in \cite{HM2}.

\begin{theo}
$x+y+1$ divides $f_t(x,y)$ if and only if $t$ is odd.
\end{theo}

\begin{proof}
We substitute $y=-x-1$ in $f_t(x,y)=(x+1)^t-x^t-(y+1)^t+y^t$ obtaining 
$(x+1)^t-x^t-(-1)^t(x)^t+(-1)^t(x+1)^t$ which is identically  zero if and only if 
$t$ is odd.
\end{proof}

Therefore, if $t$ is odd then $g_t(x,y)=\frac{f_t(x,y)}{x-y}$ has an absolutely irreducible factor 
over $\mathbb{F}_p$, so $x^t$ it is not a PN function over $\mathbb{F}_{p^n}$  for infinitely many $n$.

\vspace{.2cm}
\textbf{We will assume from now that $t$ is even.}
This implies that $\ell$ is odd. 
Moreover $\ell\geq 3$, because for $\ell=1$,
$t$ is known to be PN (see Table 1).
\vspace{.2cm}

Notice that 
\[(x+a)^t=(x^{p^i}+a^{p^i})^{\ell}(x+a)^{r}=\]
\[\bigl(x^{p^i\ell}+\binom{\ell}{1}x^{p^i(\ell-1)}a^{p^i}+\cdots+\binom{\ell}{\ell- 1} x^{p^i}a^{p^i(\ell-1)}+a^{p^i\ell}\bigr)
\times   \]
\[ \bigl(x^{r}+\binom{r}{1} x^{r-1}a+\cdots+\binom{r}{r- 1}xa^{r-1}+a\bigr).\]

To study the singularities  we expand at the point $P=(\alpha,\beta)$, 
so then we need to study
\[
f_t(x+\alpha,y+\beta)=(x+\alpha+1)^t-(x+\alpha)^t-(y+\beta+1)^t+(y+\beta)^t.
\]
We write this as a sum of homogeneous parts
\[
f_t(x+\alpha,y+\beta)=F_0+F_1+F_2+\cdots
\]
where $F_i=F_i(x,y)$ is 0 or is a homogeneous polynomial of degree $i$.
{\bf This notation is fixed for the entire paper.}
By definition, a point $P$ is a \emph{singular point}
(or singularity) if and only if  $F_0=F_1=0$ at $P$.
We compute that

\begin{eqnarray*}
F_0&=&((\alpha+1)^{p^i\ell+r}-\alpha^{p^i\ell+r}-(\beta+1)^{p^i\ell+r}+\beta^{p^i\ell+r},\cr
F_1(x,y)&=& \binom{r}{r-1}[((\alpha+1)^{p^i\ell+r-1}-\alpha^{p^i\ell+r-1})x-((\beta+1)^{p^i\ell+r-1}-\beta^{p^i\ell+r-1})y], \cr
F_2(x,y)&=& \binom{r}{r-2}[((\alpha+1)^{p^i\ell+r-2}-\alpha^{p^i\ell+r-2})x^2-((\beta+1)^{p^i\ell+r-2}-\beta^{p^i\ell+r-2})y^2], \cr
\cdots
\cr
F_u(x,y)&=&\binom{r}{r-u}[((\alpha+1)^{p^i\ell+r-u}-\alpha^{p^i\ell+r-u})x-((\beta+1)^{p^i\ell+r-u}-\beta^{p^i\ell+r-u})y],
\cr
\cdots
\cr
F_{r}(x,y)&=&((\alpha+1)^{p^i\ell}-\alpha^{p^i\ell})x-((\beta+1)^{p^i\ell}-\beta^{p^i\ell})y,
\cr
F_{p^i}(x,y) &=& jx^{p^i}((\alpha+1)^{p^i(\ell-1)+r}-\alpha^{p^i(\ell-1)+r})-jy^{p^i}((\beta+1)^{p^i(\ell-1)+r}-\beta^{p^i(\ell-1)+r}),  \cr
F_{p^i+1}(x,y) &=& \binom{r}{r-1}[jx^{p^i+1}((\alpha+1)^{p^i(\ell-1)+r-1}-\alpha^{p^i(\ell-1)+r-1}) \cr
&&-jy^{p^i+1}((\beta+1)^{p^i(\ell-1)+r-1}-\beta^{p^i(\ell-1)+r-1})].
\end{eqnarray*}

\begin{lemma}\label{pro: NoZeroF1F2}
If $F_1(x,y)=F_2(x,y)=0$ then  $r=0$ or $1$.
\end{lemma}

\begin{proof}
$F_1(x,y)=0$ implies that $(\alpha+1)^{p^i\ell+r-1}-\alpha^{p^i\ell+r-1}=0 \Leftrightarrow (\frac{\alpha+1}{\alpha})^{p^i\ell+r-1}=1$. 

$F_2(x,y)=0$ implies that $(\alpha+1)^{p^i\ell+r-2}-\alpha^{p^i\ell+r-2}=0 \Leftrightarrow (\frac{\alpha+1}{\alpha})^{p^i\ell+r-2}=1$. 

Hence 
\[
(\frac{\alpha+1}{\alpha})^{p^i\ell+r-1}/(\frac{\alpha+1}{\alpha})^{p^i\ell+r-2}=1 \Leftrightarrow (\frac{\alpha+1}{\alpha})=1 \Leftrightarrow 1=0.
\]
But this is impossible. So, the only possibilities are:
\begin{itemize}
\item $F_1(x,y)\neq 0$ and the coefficient $\binom{r}{r-2}=0$ in $F_2(x,y)$, i.e., $r=0$ or $1$.
\item  Both coefficients are zero, $\binom{r}{r-2}=\binom{r}{r-1}=0\Leftrightarrow r=0$.
\end{itemize}
This completes the proof.
\end{proof}

A point $P$ is singular iff $F_0=F_1=0$ at $P$. Thus, we need to expand the expression 
$((\alpha+1)^{p^i\ell+r-1}-\alpha^{p^i\ell+r-1})=$
\[\alpha^{p^i\ell+r-1}+j\alpha^{p^i(\ell-1)+r-1}+\cdots+j\alpha^{p^i+r-1}+\alpha^{r-1}+\]
\[\binom{r-1}{1}\alpha^{p^i\ell+r-2}+j\binom{r-1}{1}\alpha^{p^i(\ell-1)+r-2}+\cdots+j\binom{r-1}{1}\alpha^{p^i+r-2}+\binom{r-1}{1}\alpha^{r-2}+\]
\[\cdots\]
\[\binom{r-1}{r-2}\alpha^{p^i\ell+1}+j\binom{r-1}{r-2}\alpha^{p^i(\ell-1)+1}+\cdots+j\binom{r-1}{r-2}\alpha^{p^i+1}+\binom{r-1}{r-2}\alpha+\]
\[\alpha^{p^i\ell}+j\alpha^{p^i(\ell-1)}+\cdots+j\alpha^{p^i}+1 -\alpha^{p^i(\ell-1)+r}.\]

This is a complicated expression, therefore  we distinguish different cases.
Recall $t=p^i\ell+r$.
Because $x^t$ is planar iff $x^{pt}$ is planar, we may assume that
$t$ is relatively prime to $p$, i.e., $r\neq 0$.

\begin{itemize}
\item[(A)] $r\neq 1$.
\item[(B)] $r=1$.  We divide this case into two subcases:
\begin{itemize}
\item[(B.1)] $gcd(\ell, p^i-1)<\ell$. 
\item[(B.2)] $gcd(\ell, p^i-1)=\ell$. 
\end{itemize}
\end{itemize}

In this paper we will prove many cases of (A) and essentially
all of case (B).

Next we describe the singular points in the various cases.





 \subsection{Affine Singular Points in Case (A)}

The analysis is straightforward in this case.

\begin{lemma}\label{le: Mult2casoA}
If $P=(\alpha,\beta)$ is a singular point and we are in case (A), then the multiplicity is $m_P(f_t)=2$ and 
$m_P(g_t)=2$ if $\alpha\neq \beta$ and $m_P(g_t)=1$ otherwise. 
\end{lemma}

\begin{proof}
Since $P$ is a singular point we have  $F_1(x,y)=0$. 
The multiplicity must be two because otherwise 
$F_2(x,y)=0$ and then Lemma \ref{pro: NoZeroF1F2} would be false.

Moreover, if $\alpha=\beta$ then $P$ is a point on the curve
$x-y$ so the multiplicity decreases by one for $g_t(x,y)$.
\end{proof}

\begin{lemma}\label{le:Proct2LinesCaseA}
If $P=(\alpha,\beta)$ is a singular point and we are in case (A), then the homogeneous component $F_2(x,y)$ is a product of two different lines. 
\end{lemma}
\begin{proof}
Remember that 
\[
F_2(x,y)= \binom{r}{r-2}[((\alpha+1)^{p^i\ell+r-2}-\alpha^{p^i\ell+r-2})x^2-((\beta+1)^{p^i\ell+r-2}-\beta^{p^i\ell+r-2})y^2].
\]
We can rewrite this as 
\[
F_2(x,y)= \binom{r}{r-2}[(ax+by)*(ax-by)],
\]
where $a=((\alpha+1)^{p^i\ell+r-2}-\alpha^{p^i\ell+r-2})^{1/2}$ and $b=((\beta+1)^{p^i\ell+r-2}-\beta^{p^i\ell+r-2})^{1/2}$.
\end{proof}

\subsection{Affine Singular Points in Case (B)}

We always use the notation $P=(\alpha,\beta)$.
Looking at the homogeneous components in $f_t(x,y)$ we obtain  equations for the singular points.

\begin{lemma}
$P=(\alpha,\beta)$ is a singular point of $f_t$ if and only if 
\begin{equation}\label{eq:1}
 (\alpha+1)^{p^i \ell+1}-\alpha^{p^i \ell+1} -(\beta+1)^{p^i \ell+1}+\beta^{p^i \ell+1}=0
\end{equation}
\begin{equation}\label{eq:2}
 (\alpha+1)^{\ell}-\alpha^{ \ell}=0;  (\beta+1)^{\ell}-\beta^{\ell}=0
\end{equation}
\end{lemma}

It follows from it  that $f_t(x,y)$ has at most $(\ell-1)^2$ singular points.

We need to compute the multiplicity of the singular points. The homogeneous component $F_{p^i}(x,y)$ is non-zero except when 
\begin{equation}\label{eq:3}(\alpha+1)^{p^i(\ell-1)+1}-\alpha^{p^i(\ell-1)+1}=0.
\end{equation}
 and 
 \begin{equation}\label{eq:4}(\beta+1)^{p^i(\ell-1)+1}-\beta^{p^i(\ell-1)+1}=0.
 \end{equation}
 
 Hence we conclude the following.

\begin{lemma}
Let $P=(\alpha,\beta)$ be a singular point of $f_t$. 
Then $F_{p^i} =0$ if and only if  equations  (\ref{eq:1}) and (\ref{eq:2}) hold together with 
\begin{equation}\label{eq:5}
\alpha^{\ell}=\beta^{\ell}
\end{equation}
\begin{equation}\label{eq:6}
\alpha^{p^i-1}=\beta^{p^i-1}
\end{equation}
\end{lemma}
\begin{proof}
From (\ref{eq:1}) and (\ref{eq:2}) we clearly have that $\alpha^{\ell}=\beta^{\ell}$. 

Using (\ref{eq:3}) and (\ref{eq:4}) in (\ref{eq:1}) we obtain   
$$
\alpha^{p^i (\ell-1)+1}[(\alpha+1)^{p^i}-\alpha^{p^i}]-\beta^{p^i (\ell-1)+1}[(\beta+1)^{p^i }-\beta^{p^i}]=0,
$$
which is equal to  
$$
\alpha^{p^i (\ell-1)+1}-\beta^{p^i (\ell-1)+1}=0.
$$
Multiplying it by $\alpha^{p^i}\beta^{p^i}$ we get, 
$$
\alpha^{p^i \ell+1}\beta^{p^i}-\beta^{p^i \ell+1}\alpha^{p^i}=0.
$$
Since  $\alpha^{\ell}=\beta^{\ell}$ the latter is equivalent to 
$\alpha \beta^{p^i}=\beta\alpha^{p^i}=0 \Leftrightarrow \alpha^{p^i-1}=\beta^{p^i-1}.$
\end{proof}

Thus  a singular point $(\alpha,\beta)$ has homogeneous component $F_{p^i}=0$ if and only if $\alpha=\beta E$ where $E^d=1$, and $d=gcd(\ell, p^i-1)$. 

How many  such  points there are is the next issue.

\begin{lemma}
In case (B) there are at most $(d-1)^2$ singular points $(\alpha,\beta)$ such that $F_{p^i} =0$,
where $d=gcd(\ell, p^i-1)$.
\end{lemma}
\begin{proof}
$F_{p^i}(x+\alpha,y+\beta)=0$ if and only if 
\[\alpha^{\ell}=\beta^{\ell}\]
\[
 (\alpha+1)^{\ell}-\alpha^{ \ell}=0;  (\beta+1)^{\ell}-\beta^{\ell}=0\]
\[(\alpha+1)^{p^i(\ell-1)+1}-\alpha^{p^i(\ell-1)+1}=0.\]
\[(\beta+1)^{p^i(\ell-1)+1}-\beta^{p^i(\ell-1)+1}=0.\]

Equivalently
\[(\alpha/\beta)^{\ell}=1\]
\[
 (1+1/\alpha)^{\ell}=1;  (1+1/\beta)^{\ell}=1\]
\[(1+1/\alpha)^{p^i(\ell-1)+1}=1.\]
\[(1+1/\beta)^{p^i(\ell-1)+1}=1.\]
Doing the change of coordinates $1/\alpha=a$ and  $1/\beta=b$ we get
\[(b/a)^{\ell}=1\]
\[
 (1+a)^{\ell}=1;  (1+b)^{\ell}=1\]
\[(1+a)^{p^i(\ell-1)+1}=1.\]
\[(1+b)^{p^i(\ell-1)+1}=1.\]

Doing the change of coordinates $a=a_1-1$ and  $b=b_1-1$ we get

\[(b_1-1)^{\ell}=(a_1-1)^{\ell}\]
\[
 (a_1)^{\ell}=1;  (b_1)^{\ell}=1\]
\[(a_1)^{p^i(\ell-1)+1}=1.\]
\[(b_1)^{p^i(\ell-1)+1}=1.\]

Moreover for any pair $(a_1,b_1)$ satisfying this equations with $a_1,b_1\neq 1$ there is a singular point $(\alpha,\beta)$ with $F_{p^i}(x+\alpha,y+\beta)=0$.

Notice that $$(a_1)^{\ell}/(a_1)^{p^i(\ell-1)+1}=a_1^{p^i-1}=1$$ and 
$$(b_1)^{\ell}/(b_1)^{p^i(\ell-1)+1}=b_1^{p^i-1}=1.$$

Therefore $a_1^d=1=b_1^d$, where $d=gcd(\ell,p^i-1)$. So we conclude that there is at most $(d-1)^2$ singular points $(\alpha,\beta)$ with $F_{p^i} =0$.
\end{proof}

\begin{coro}
In case (B.1) there are at most $(\ell/3-1)^2$ singular points $(\alpha,\beta)$ with  $\alpha\neq \beta$ and  $F_{p^i} =0$.
\end{coro}
\begin{proof}
Since $t=p^i\ell+1$ is even then $\ell$ is odd, so $d$ is odd as well.  Since $gcd(\ell, p^i-1)<\ell$ both odd then $d\leq \ell/3$. Hence we have at most $(\ell/3-1)^2$ points. 
\end{proof}

Finally, we have a lemma concerning the component of degree $p^i+1$.
We will come back to use this in Section \ref{furtheranalysis}.

\begin{lemma}
$F_{p^i+1}=A x^{p^i+1}- B y^{p^i+1}$, where $A=(\alpha+1)^{p^i(\ell-1)}-\alpha^{p^i(\ell-1)}$ and 
$B=(\beta+1)^{p^i(\ell-1)}-\beta^{p^i(\ell-1)}$.  
\end{lemma}

So the multiplicity is at most $p^i+1$, and it is actually $p^i+1$ unless equations (\ref{eq:5}) and (\ref{eq:6}) hold.

\subsection{Singular Points at Infinity}

Next we consider singular points at infinity.
The projective curve we are working with is $f_t(x,y,z)=[(x+z)^t-x^t-(y+z)^t+y^t]/z$. 
After cancelling terms we get, $f_t(x,y,z)=tx^{t-1}-ty^{t-1}+z(\text{lower order terms})$. 
The next result explains when there are singular points at infinity (the chart where $z=0$).

\begin{lemma}\label{le:NoSingularCaseA}
There are no singular points at infinity in case  (A). 
In case (B.1) $(\alpha, 1, 0)$ is a singular point of 
$f_t(x,y,z)$  if and only if  $\alpha^{t-1}=1$, which is equivalent to  $\alpha^\ell=1.$
\end{lemma}

\begin{proof}
 The dehomogenization of  $f_t(x,y,z)=[(x+z)^t-x^t-(y+z)^t+y^t]/z$ relative to $y$ is $f'_t(x,z)=f_t(x,1,z)/z$.
We compute
 \begin{eqnarray*}
f'_t(x+\alpha,z) &=&[(x+z+\alpha)^t-(x+\alpha)^t-(z+1)^t+1]/z \cr
&=&1/z\bigl[(x+z)^t+\binom{t}{t-1}(x+z)^{t-1}\alpha+\cdots+\binom{t}{1}(x+z)\alpha^{t-1}+\alpha^{t}\\
&-&x^t-\binom{t}{t-1}(x)^{t-1}\alpha-\cdots-\binom{t}{1}(x)\alpha^{t-1}-\alpha^{t}\\
&-&z^t-\binom{t}{t-1}(z)^{t-1}-\binom{t}{t-2}(z)^{t-2}-\cdots-\binom{t}{1}(z)-1 +1\bigr]\\
&=&\binom{t}{1}(\alpha^{t-1}-1)+ \binom{t}{2}[(\alpha^{t-2}-1)z+2\alpha^{t-2}x]\\
&&+ \text{higher order terms.}
\end{eqnarray*}
Notice that the linear part cannot be zero, unless the coefficient $\binom{t}{2}=0$, which implies $r= 0$ or $1$.  Therefore, there are no singular points in  case (A).
And if $r=1$, then $P$ is a singular point iff $\alpha^{t-1}-1=0$.
 \end{proof}

The $'$ notation always refers to expansions at infinity.

\begin{lemma}\label{le: HomogeneousAtInfinity}
In case (B) let  $(\alpha,1,0)$ be a singular point of $g'_t(x,z)$   at infinity.  Then:
$$F'_{p^i-1}=z^{p^i-1}(\alpha^{p^i(\ell-1)+1}-1).
$$
$$F'_{p^i}=z^{p^i}(\alpha^{p^i(\ell-1)}-1)+(xz^{p^i-1}+x^{p^i})\alpha^{p^i(\ell-1)}.$$
\end{lemma}

\begin{proof}

$$f_t(x+\alpha+1,z)=1/z[(x+z+\alpha)^t-(x+\alpha)^t-(z+1)^t+1]=$$
$$1/z[((x+z)^{p^i\ell+1}+j(x+z)^{p^i(\ell-1)+1}\alpha^{p^i}+\cdots+j(x+z)^{p^i+1}\alpha^{p^i(\ell-1)}+(x+z)\alpha^{p^i\ell}+$$
$$(x+z)^{p^i\ell}\alpha+j(x+z)^{p^i(\ell-1)}\alpha^{p^i+1}+\cdots+j(x+z)^{p^i}\alpha^{p^i(\ell-1)+1}+\alpha^{p^i\ell+1}$$
$$-(x)^{p^i\ell+1}-j(x)^{p^i(\ell-1)+1}\alpha^{p^i}-\cdots-j(x)^{p^i+1}\alpha^{p^i(\ell-1)}-(x)\alpha^{p^i\ell}$$
$$-(x)^{p^i\ell}\alpha-j(x)^{p^i(\ell-1)}\alpha^{p^i+1}-\cdots-j(x)^{p^i}\alpha^{p^i(\ell-1)+1}-\alpha^{p^i\ell+1}$$
$$-(z)^{p^i\ell+1}-j(z)^{p^i(\ell-1)+1}-\cdots-j(z)^{p^i+1}-(z)$$
$$-(z)^{p^i\ell}-j(z)^{p^i(\ell-1)}-\cdots-(z)^{p^i}-1+1)].$$

Therefore we get:
$$F'_{p^i-1}=jz^{p^i-1}(\alpha^{p^i(\ell-1)+1}-1).
$$
$$F'_{p^i}=jz^{p^i}(\alpha^{p^i(\ell-1)}-1)+j(xz^{p^i-1}+x^{p^i})\alpha^{p^i(\ell-1)}.$$
\end{proof}

\begin{lemma}
In case (B) let  $(\alpha,1,0)$ be a singular point of $g'_t(x,z)$   at infinity.  Then the multiplicity at $P$ is $p^i$ if $\alpha\in \mathbb{F}_{p^i}$, and is $p^i-1$ otherwise.
\end{lemma}

\begin{proof}

Since $P=(\alpha,1,0)$ is a singular point at the infinity we know that $\alpha^{\ell}=1$. Moreover $F_{p^i}=0 $ iff $\alpha^{p^i(\ell-1)+1}-1=0$. 
Multiplying the last equation by $\alpha^{p^i}$ we equivalently obtain $\alpha^{p^i(\ell)+1}-\alpha^{p^i}=0\Leftrightarrow \alpha^{p^i-1}=1 \Leftrightarrow
\alpha\in \mathbb{F}_{p^i}$. 
This completes the proof.
\end{proof}

\subsection{The Multiplicities}

Next we pin down the multiplicities of these singular points
$P=(\a,\b)$ on $f_t(x,y)$, and how things change for $g_t(x,y)$.
Recall that the $'$ notation always refers to expansions at infinity.

We classify the points into the following types:
\begin{itemize}
\item[(I)]  $\alpha^{d}=\beta^{d}$, i.e., $F_{p^i}=0$.
\item[(II)] $\alpha^{d}\neq \beta^{d}$, i.e., $F_{p^i}\neq 0$.
\item[(III)] $(\alpha,1,0)$ is a singular point at  infinity.
\begin{itemize}
\item[(III.A)] $\alpha\in \mathbb{F}_{p^i}$, i.e., $F'_{p^i-1}=0$.
\item[(III.B)]$\alpha\notin \mathbb{F}_{p^i}$, i.e., $F'_{p^i-1}\neq 0$.
\end{itemize}
\end{itemize}

We note that
if $\ell=1$ and $i>1$, the only singular point is $(-2,-2)$.

Defining $w(x,y):=x-y$ we note the following
multiplicities on $w$:  $m_P(w)=1$ if $P=(\alpha,\alpha)$ (Type II), 
and $m_P(w)=0$ for all other singular points $P=(\a,\b)$.
We now have the multiplicities for $g_t(x,y)$.
 
\begin{center}
\begin{tabular}{|l | c| c | c|}
\hline
Type & Number of Points & $m_P(f_k)$ & $m_P(g_t)$ \\
\hline
I & $N_1 \leq (\ell/3-1)^2$  if $gcd(\ell,p^i-1)< \ell$ & $p^{i}+1$ &  $\leq p^{i}+1$\\
& $N_1 \leq (\ell-1)^2$  if $gcd(\ell,p^i-1)= \ell$ &  & \\
\hline
II & $ \leq (\ell-1)^2-N_1$ & $ p^i$ & $\leq p^i$\\
\hline
III.A & $N_2 \leq \ell$ & $ p^i$ & $\leq p^i$\\
\hline
III.B & $\ell-N_2$ & $ p^i-1$ & $\leq p^i-1$\\
\hline
\end{tabular}
\end{center}

\subsection{Further Analysis}\label{furtheranalysis}

Next we need some further analysis of the singularities and the
homogeneous components.

The following result is clear, because we are in characteristic $p$.

\begin{lemma}\label{repeatedLine}
$F_{p^i}=(\s x-\t y)^{p^i}$ where 
$\s^{p^i}=((\alpha+1)^{p^i(\ell-1)+1}-\alpha^{p^i(\ell-1)+1})$
and $\t^{p^i}=((\beta+1)^{p^i(\ell-1)+1}-\beta^{p^i(\ell-1)+1})$. $F'_{p^i}=(Uz)^{p^i}$ where $U^{p^i}=j(\alpha^{p^i(\ell-1)+1}-1)$.
\end{lemma}

\begin{lemma}\label{differentfactors}
$F_{p^i+1}$ (resp $F'_{p^i}$) consists of $p^i+1$ (resp $p^i$) different linear factors.
\end{lemma}
\begin{proof}
We have seen that $F_{p^i+1}=A x^{p^i+1}-B y^{p^i+1}$. 
Consider $h(x)=F_{p^i+1}(x,1)=A x^{p^i+1}-B$. If $h(x)$ has a repeated root at $a$ then $h'(a)=0$. Consider the derivative $h'(x)=A x^{p^i}$ which is never zero only for $x=0$ which is not a root of $h$ therefore there are no
repeated factors in $F_{p^i+1}$. Same arguments may be used to prove the result for $F'_{p^i}$.
\end{proof}

Next we make a crucial observation for our proofs.
For the rest of this paper, we let $L=\s x-\t y$, so that $F_{p^i}=L^{p^i}$.
Suppose $g_t(x,y)=u(x,y)v(x,y)$, and suppose that the Taylor expansion at
a singular point $P=(\a,\b)$ is
\[
u(x+\a, y+\b)=L^{r_1}+u_1, \  v(x+\a,y+\b)=L^{r_2}+v_1
\]
where wlog $r_1\leq r_2$.
Then $F_{p^{i}+1}=L^{r_1}(v_1+L^{r_2-r_1}u_1)$. 
From Lemma \ref{differentfactors} we deduce the following.

\begin{lemma}\label{multp0o1}
With the notation of the previous paragraph,
\begin{itemize}
\item[(i)] Either $r_1=1$ or $r_1=0$.
\item[(ii)]  If $r_1=1$ then $gcd(L,v_1+L^{r_2-r_1}u_1)=1$.
\end{itemize}
\end{lemma}

 \bigskip

We next make two quick remarks to aid us in moving between
$f_t(x,y)$ and $g_t(x,y)$.
Suppose that $P=(\a,\b)\neq(1,1)$ is a singular point of $g_t(x,y)$
such that $F_{p^i}(x,y)\neq0$ at $P$ (type II).
We will need to know the greatest
common divisor $(G_{m}(x,y),G_{m+1}(x,y))$ where $m=m_P(g_t)$.
This can be found from $(F_{p^i}(x,y),F_{p^i+1}(x,y))$ as follows.

Again letting $w(x,y)=x-y$,
we have
            $$f_t(x+\a,y+\b)=w(x+\a,y+\b)g_t(x+\a,y+\b),$$
and so
$$         F_{p^i}(x,y)+F_{p^i+1}(x,y)+\cdots
          = (W_0+W_1(x,y))(G_m(x,y)+G_{m+1}(x,y)+\cdots).
$$
where polynomials with subscript $i$ are 0 or homogeneous of degree $i$.

\begin{remark}\label{remw1}
We get:

If $W_0\neq 0$, i.e., $\a\neq\b$ 
\begin{eqnarray}
F_{p^i}&=&W_0 G_{p^i} = (\sigma x+\tau y)^{2^i} \cr
F_{p^i+1}&=&W_1 G_{p^i}+W_0 G_{p^i+1} 
\end{eqnarray}
then it follows from these equations that
$(F_{2^i},F_{2^i+1})=(G_{2^i},G_{2^i+1})$.

If $W_0= 0$, i.e., $\a=\b$, 
\begin{eqnarray}
F_{p^i}&=&W_1 G_{p^i-1} = (\sigma x+\tau y)^{2^i} \cr
F_{p^i+1}&=&W_1 G_{p^i}  .
\end{eqnarray}

it is clear that (up to scalars) $W_1 =\sigma x-\tau y$,
and so $(F_{p^i},F_{p^i+1})= \sigma x-\tau y$ because
$F_{p^i+1}(x,y)$ has distinct linear factors (Lemma \ref{differentfactors}).
Hence $(G_{p^i-1},G_{p^i})=1$. 
The same result is true for the points at infinity of type $III.B$, i.e.,
$(G'_{p^i-1},G'_{p^i})=1$.
\end{remark}

\section{Intersection Multiplicity}\label{anaintersect}

In this section we  compute 
and give bounds on the intersection multiplicities.

\begin{lemma}\label{GCD=1}
 Let $h(x,y)$ be an affine curve.
Write $h(x+\alpha,y+\beta)=H_m+H_{m+1}+\cdots$ where
$P=(\alpha,\beta)$ is a point on  $h(x,y)$ of multiplicity $m$.
Suppose that  $H_m$ and $H_{m+1}$  are relatively prime,
and that there is only one tangent direction at $P$.
If $h=uv$ is reducible, then $I(P,u,v)=0$.
 \end{lemma}
 
Proof: See \cite{Janwa-McGuire-Wilson}.

\subsection{Type I}

We upper bound the intersection multiplicity at the Type I point.

\begin{lemma}
If $g_t(x,y)=u(x,y)v(x,y)$ and $P$ is of Type I then
$I(P,u,v)\leq  (\frac{p^{i}+1}{2})^2$.
\end{lemma}

Proof:
Let $P$ be of Type I.
We know that $m_P(g_t)=p^i+1=m_P(u)+m_P(v)$.
From Lemma \ref{differentfactors} we know that
$F_{p^i+1}$ has $p^i+1$ different linear factors. 
Thus, $I(P,u,v)=m_p(u)m_p(v)$. 
This quantity is maximized when  $m_P(u)=m_P(v)$ and in this case
 $m_p(u)m_p(v)=(\frac{p^{i}+1}{2})^2$.

\subsection{Type II}

We show that there are two possibilities for the intersection multiplicity at a Type II point.

\begin{lemma}\label{IntersectionTypeII}
Suppose we are in case (B).
If $g_t(x,y)=u(x,y)v(x,y)$ and $P=(\alpha,\beta)$ is a point of type (II) then
either $I(P,u,v)= p^i$ or $I(P,u,v)= 0$.
\end{lemma}

Proof:
Assume $g_t(x,y)=u(x,y)v(x,y)$. 
Since $P$ is not on $w(x,y)=x-y$ 
by Lemma \ref{multp0o1} we know that 
$m_P(u)$ is either $1$ or $0$. If $m_P(u)=0$ then $I(P,u,v)= 0$.
If $m_P(u)=1$ we proceed as follows.

Let $L(x,y)=\s x + \t y$ and suppose (from the proof of Lemma \ref{multp0o1})
we have the following Taylor expansions at $P$:
$$
u(x+\a, y+\b)=L(x,y)+U_2(x,y)+\cdots
$$
$$
v(x+\a, y+\b)=L(x,y)^{p^i-1}+V_{p^{i}}(x,y)+\cdots
$$
It follows that
$$
u(x+\a,y+\b)L(x,y)^{p^i-2}-v(x+\a,y+\b)=L(x,y)^{p^i-2}U_2(x,y)-V_{p^{i}}(x,y)+\cdots.
$$
By definition of intersection multiplicity we have 
\[
I(P,u,v)=I(0,u(x+\a,y+\b),u(x+\a,y+\b)L(x,y)^{p^i-2}-v(x+\a,y+\b))
\] 
so we compute the right-hand side.   
Notice that $L(x,y)\nmid L(x,y)^{p^i-2}U_2(x,y)-V_{2^{i}}(x,y)$ because if $L(x,y)$ 
divides $L(x,y)^{p^i-2}U_2(x,y)-V_{p^{i}}(x,y)$ then $L(x,y)$ also divides $V_{p^{i}}(x,y)$. Hence,  $L(x,y)^2$ divides $L(x,y)(L(x,y)^{p^i-2}U_2(x,y)+V_{p^{i}}(x,y))=G_{p^i+1}(x,y)$ which is a contradiction.

 Therefore,  $u(x+\a,y+\b)$ and $u(x+\a,y+\b)L^{p^i-2}-v(x+\a,y+\b)$ have different tangent cones.
It follows from a property of $I(P,u(x,y),v(x,y))$ that 
\[
I(0,u(x+\a,y+\b),u(x+\a,y+\b)L^{p^i-2}-v(x+\a,y+\b))=
\]
\[
 m_0(u(x+\a,y+\b))m_0(u(x+\a,y+\b)L^{p^i-2}-v(x+\a,y+\b))=p^i.
\]

\subsection{Type III}
Next result is equivalent to Type I with using multiplicity $p^i$ instead. 
\begin{lemma}\label{IntersectionTypeIII.A}
If $g_t(x,y)=u(x,y)v(x,y)$ and $P=(\alpha,\beta)$ is a point of type (III.A) then
$I(P,u,v)\leq p^{2i}/4$.
\end{lemma}

We show that intersection multiplicities at Type III.B points are 0, so these
points may be disregarded.

\begin{lemma}\label{IntersectionTypeIII.B}
If $g_t(x,y)=u(x,y)v(x,y)$ and $P=(\alpha,\beta)$ is a point of type (III.B) then
$I(P,u,v)=0$.
\end{lemma}

Proof:
Notice that $gcd(F'_{p^i-1},F'_{p^i})=1$ and therefore
$gcd(G'_{p^i-1},G'_{p^i})=1$ by Remark \ref{remw1}. 
The proof concludes using Lemma \ref{GCD=1}.

\subsection{Case A}

\begin{lemma}
If $P$ is a singular affine  point,  and we are in case (A), 
and $g_t=g_1\cdots g_r$ over the algebraic closure $\overline{\mathbb{F}_p}$, then there are at most two indices $i,j$ such that $g_i(P)=g_j(P)=0$ with multiplicity $m_P(g_i)=m_P(g_j)=1$. Moreover $I(P,g_i,g_j)\leq 1$.
\end{lemma}

\begin{proof}
From Lemma \ref{le: Mult2casoA} we know that  $m_P(g_t)\leq 2$ therefore there are at most 2 factors of $g_t$ containing $P$, otherwise it would have more multiplicity. Moreover from Lemma \ref{le:Proct2LinesCaseA} we know that the tangent cone of $g_t$ at $P$ is a pair of different lines, then $I(P,g_i,g_j)=m_p(g_i)m_p(g_j)\leq 1$. 
\end{proof}

\section{Case B: Proof assuming  $g_t(x,y)$ irreducible over $\mathbb{F}_p$}\label{part2proof}

Here is a well known result, its proof can
be found in \cite{Kopparty-Yekhanin}.

\begin{lemma}\label{IfIrreducibleEqualDegreeFactors}
Suppose that $p(\underline{x})\in \mathbb{F}_q[x_1,\ldots,x_n]$ is of degree $t$ and is irreducible in $\mathbb{F}_q[x_1,\ldots,x_n]$. Let $\mathbb{K}$ be a finite field extension of $\mathbb{F}_q$. Then there exists a $\mathbb{K}$-irreducible polynomial $h(\underline{x})$ such that the factorization of $p(\underline{x})$ into $\mathbb{K}$-irreducible polynomials is
$$
p(\underline{x})=c\prod_{\sigma\in G}\sigma(h(\underline{x})),
$$
where $G=Gal(\mathbb{K}/\mathbb{F}_{q})$ and $c\in \mathbb{F}_{q}$. Furthermore if $p(\underline{x})$
is homogeneous, then so is $h(\underline{x})$. 

In particular, there exists $r\mid t$ and an absolutely irreducible polynomial
$m(\underline{x})\in\mathbb{F}_{q^r}[x_1,\ldots,x_n]$ of degree $\frac{t}{r}$ such that
$$
p(\underline{x})=d\prod_{\sigma\in G'}\sigma(m(\underline{x})),
$$
where $G'=Gal(\mathbb{F}_{q^r}/\mathbb{F}_{q})$ and $d\in \mathbb{F}_{q}$.
\end{lemma}


\begin{remark}\label{samedegree}
Notice that if $u(x,y)=\sum a_{i,j}x^i y^j$ then $\sigma(u(x,y))=\sum \sigma(a_{i,j})x^i y^j$ where $\sigma\in G$
is  the Frobenius map (or a power of it). Therefore, $u$ and $\sigma(u)$ have the same monomials and only differ in some coefficients. This means that both  $u$ and $\sigma(u)$ have the same degree.
\end{remark}

\begin{theo}
Suppose we are in case (B.1) and that $g_{t}(x,y)$ is irreducible over $\mathbb{F}_p$.
Then $g_t(x,y)$ is absolutely irreducible whenever:
\begin{itemize}
\item either $p\geq 5 $, $i\geq 1$ and $\ell> 1$,
\item or $p=3 $, $i\geq 2$ and $\ell> 1$.
\end{itemize}
\end{theo}
\begin{proof}
Suppose not, then $g_t(x,y,z)=u(x,y,z)v(x,y,z)$. Using Remark \ref{samedegree} we have that $deg(u)=deg(v)=(p^{i}\ell-1)/2$ .
We apply Bezout's Theorem to $u$ and $v$:
\begin{equation}\label{eqB}
\sum_{P\in Sin(g)} I(P,u,v)=deg(u) deg(v)=((p^{i}\ell-1)/2)^2.
\end{equation}
We can bound the left hand side as follows,
\begin{equation}\label{eqM}
\sum_{P\in Sing(g_t)} I(P,u,v)=\sum_{P\in I} I(P,u,v)+\sum_{P\in II} I(P,u,v)+\sum_{P\in III.A} I(P,u,v)
\end{equation}

\begin{equation}\label{eqMbound}
\leq  \frac{(p^i+1)^2}{4}(\ell/3-1)^2 +\frac{4}{4}p^i[(\ell-1)^2-(\ell/3-1)^2])+\frac{p^{2i}}{4}\ell
\end{equation}

We have that (\ref{eqB})$\leq$(\ref{eqMbound}) by Bezout's Theorem, 
so if we prove that (\ref{eqB})$>$(\ref{eqMbound}) we get a contradiction.
The inequality (\ref{eqB})$>$(\ref{eqMbound}) is
\begin{equation}\label{eq:ecuacion-clave}
p^{2i}\ell^2-2p^i\ell+1> (p^{2i}+2p^i+1)(\ell^2/9-2/3\ell+1)+4p^i(\ell^2-2\ell+1-\ell^2/9+2/3\ell-1)+p^{2i}\ell
\end{equation}

\begin{equation}\label{eq:ecuacion-clave1}
p^{2i}(\ell^2-\ell^2/9+2/3\ell-1-\ell)> p^i(2\ell-2(\ell^2/9-2/3\ell+1)+4(\ell^2-2\ell+1))+(\ell^2/9-2/3\ell)
\end{equation}
\begin{equation}\label{eq:ecuacion-clave2}
p^{2i}(\frac{8}{9}\ell^2-\frac{1}{3}\ell-1)> p^i(\frac{34}{9} \ell^2-\frac{14}{3}\ell+2)+(\ell^2/9-2/3\ell)
\end{equation}

\begin{equation}\label{eq:ecuacion-clave3}
1> \frac{(\frac{34}{9} \ell^2-\frac{14}{3}\ell+2)}{p^{i}(\frac{8}{9}\ell^2-\frac{1}{3}\ell-1)}+\frac{(\ell^2/9-2/3\ell)}{p^{2i}(\frac{8}{9}\ell^2-\frac{1}{3}\ell-1)}
\end{equation}

Notice that in one hand we have that 
$$
 \frac{(\frac{34}{9} \ell^2-\frac{14}{3}\ell+2)}{(\frac{8}{9}\ell^2-\frac{1}{3}\ell-1)}\leq \frac{34}{8}\Leftrightarrow 
(\frac{34}{9} \ell^2-\frac{14}{3}\ell+2) \leq (\frac{34}{9}\ell^2-\frac{34}{24}\ell-34/8)  
$$
$$
\Leftrightarrow 78/24 \ell \geq 50/8 \Leftrightarrow \ell \geq 150/78 \quad \ell\geq 3.
$$
In the other hand we have that 

$$\frac{(\ell^2/9-2/3\ell)}{(\frac{8}{9}\ell^2-\frac{1}{3}\ell-1)}<\frac{1}{8} \Leftrightarrow (\ell^2/9-2/3\ell)<(\frac{1}{9}\ell^2-\frac{1}{24}\ell-\frac{1}{8}) 
$$
$$
\Leftrightarrow  \frac{15}{24}\ell>\frac{1}{8}\Leftrightarrow \ell>\frac{1}{5} \quad \ell\geq 1
$$

So the right hand side in equation (\ref{eq:ecuacion-clave3}) is less than 
$\frac{34}{8p^i}+\frac{1}{8p^{2i}}$ it is less than one if either 
$p>4$ for any $i\geq 1$ or $p=3$ for any $i\geq 2$.
 
\end{proof}


\begin{theo}
Suppose we are in case (B.2)
but $\ell< p^i-1$, and suppose that $g_{t}(x,y)$ is irreducible over $\mathbb{F}_p$ then $g_t(x,y)$ is absolutely irreducible.
\end{theo}
\begin{proof}
First of all notice that since $\gcd(\ell,p^i-1)= \ell< p^i-1$ then 
$\ell\mid p^i-1$, hence $p^i-1\geq 2\ell \Leftrightarrow p^i\geq 2(\ell-1)+3$.

Suppose not, then $g_t(x,y,z)=u(x,y,z)v(x,y,z)$. Using Remark \ref{samedegree} we have that $deg(u)=deg(v)=(p^{i}\ell-1)/2$ .
We apply Bezout's Theorem to $u$ and $v$:
\begin{equation}\label{eq1'}
\sum_{P\in Sin(g)} I(P,u,v)=deg(u) deg(v)=((p^{i}\ell-1)/2)^2.
\end{equation}
We can bound the left hand side as follows,
\begin{equation}\label{eq2'}
\sum_{P\in Sing(g_t)} I(P,u,v)=\sum_{P\in I} I(P,u,v)+\sum_{P\in II} I(P,u,v)+\sum_{P\in III.A} I(P,u,v)
\end{equation}
Since $d=\ell$ then all the singular points are of type $I$ and no points of type $II$. Between the points of type $I$ we have $(\ell-1)(\ell-2)$ with both coordinates different and $(\ell-1)$ with both coordinates equal and therefore intersection multiplicity at most $p^{2i}/4$. 

\begin{equation}\label{eq3'}
(\ref{eq2'})\leq \frac{(p^{2i}+2p^i+1)(\ell-1)(\ell-2)}{4}+\frac{p^{2i}(\ell-1)}{4}+
\frac{p^{2i}\ell}{4}
\end{equation}
We have that (\ref{eq1'})$\leq$(\ref{eq3'}) by Bezout's Theorem, 
so if we prove that (\ref{eq1'})$>$(\ref{eq3'}) we get a contradiction.

\begin{equation}\label{eq:ecuacion-clave'}
p^{2i}\ell^2-2p^i\ell+1> (p^{2i}+2p^i+1)(\ell^2-3\ell+2)+p^{2i}(\ell-1)+p^{2i}\ell
\end{equation}

\begin{equation}\label{eq:ecuacion-clave1'}
p^{2i}(\ell-1)> 2p^i(\ell^2-2\ell+2)+(\ell^2-3\ell+2)=2p^i(\ell-1)^2+2p^i+(\ell^2-3\ell+2)
\end{equation}

\begin{equation}\label{eq:ecuacion-clave2'}
p^{i}> 2(\ell-1)+\frac{2}{\ell-1}+\frac{\ell-2}{p^i}
\end{equation}
Remember that  $p^i\geq 2(\ell-1)+3$ and $\ell\geq 3$ then .
$$
2(\ell-1)+\frac{2}{\ell-1}+\frac{\ell-2}{p^i}\leq
2(\ell-1)+1+\frac{1}{2}< 2(\ell-1)+3\leq p^i.
$$
 
\end{proof}

\begin{remark}
Notice that the latter proof does not hold for $\ell=1$ as it should be, because it is already known that  $t=p^i+1$ is a exceptional number. 
\end{remark}

\begin{remark}
Only left to prove the case $\ell=p^i-1$, i.e., $t=p^{2i}-p^i+1$.
\end{remark}

\section{Case B: Proof assuming $g_t(x,y)$ not irreducible over $\mathbb{F}_p$}\label{part3proof}

Suppose $g_t=f_1\cdots f_r$ is the factorization into irreducible factors over $\mathbb{F}_p$.
Let $f_j=f_{j,1}\cdots f_{j,n_j}$ be the factorization of $f_j$ into $n_j$ absolutely irreducible factors.
Each $f_{j,s}$ has degree $\deg (f_j)/n_j$.

\begin{lemma}\label{OnlyTwohasMultiplicity}
If $P$ is a point of type $II$ then one of the following  holds:
\begin{enumerate}
\item  $m_P(f_{j,s})=0$ for all $j\in \{1,\ldots,r\}$ and $s\in\{1,\ldots,n_j\}$
except for a  pair $(j_1,s_1)$  with  $m_P(f_{j_1,s_1})=p^i$.
\item
$m_P(f_{j,s})=0$ for all $j\in \{1,\ldots,r\}$ and $s\in\{1,\ldots,n_j\}$
except for two  pair $(j_1,s_1)$ and $(j_2,s_2)$ with $m_P(f_{j_1,s_1})=1$ and $m_P(f_{j_2,s_2})=p^{i}-1$.
\end{enumerate}
\end{lemma}
\begin{proof}
This is a consequence of Lemma \ref{repeatedLine} and Lemma \ref{multp0o1}.
Consider $u=f_{a,b}$ and $v=\prod_{j\neq a, s\neq b}f_{j,s}$ from Lemma \ref{multp0o1}
we know that $m_P(f_{a,b})$ is either $0$ or $1$ or $p^{i}-1$ or $p^i$ (resp $m_p(v)$ is either $p^i$ or $p^i-1$ or $1$ or $0$). But this is true for any pair $(a,b)$.

Clearly no two components $f_{a,b}$ and $f_{a',b'}$ has multiplicity greater than or equal to $p^i-1$ because the total multiplicity $m_P(g_t)=p^i$. And there are no two components $f_{a,b}$ and $f_{a',b'}$ with multiplicity equal to $1$,
because then $u=f_{a,b}f_{a',b'}$ has $L$ two times in the tangent cone and $v=g/u$ has $L^{p^{i-2}}$ in the tangent cone which is impossible. Hence the only possibilities are:
\begin{itemize}
 \item[(i)] There exists $(a,b)$ with $m_P(f_{a,b})=p^i$, and  $m_P(f_{j,s})=0$ for
$(j,s)\neq(a,b) $.
\item[(ii)] There exist $(a,b)$ and $(a',b')$ with $m_P(f_{a,b})=1$ and $m_P(f_{a',b'})=p^i-1$, and $m_P(f_{j,s})=0$ for
$(j,s)\neq(a,b) $ , $(j,s)\neq(a',b') $.
\end{itemize}
\end{proof}

\begin{lemma}\label{IntersectionReducibleI}
If $P$ is a point of type $I$ or $III.A$, then for any two components $f_{a,b}$ and $f_{a',b'}$ we have that $I(P,f_{a,b},f_{a',b'})=m_P(f_{a,b})m_P(f_{a',b'})$.
\end{lemma}
\begin{proof}
From Lemma \ref{differentfactors} the tangent cones of $f_{a,b}$ and $f_{a',b'}$ has no common factors.
\end{proof}

\begin{lemma}\label{IntersectionReducibleIII}
If $P$ is a point of type $III.B$, then for any two components $f_{a,b}$ and $f_{a',b'}$ we have that $I(P,f_{a,b},f_{a',b'})=0$.
\end{lemma}
\begin{proof}
Consider $u=f_{a,b}$ and $v=g_m/u$. From Lemma \ref{IntersectionTypeII} we know that $I(P,u,v)=0=\sum_{(j,s)\neq (a,b)} I(P,u,f_{j,s})$, then $I(P,f_{a,b},f_{a',b'})=0$.
\end{proof}

\begin{lemma}\label{IntersectionReducibleII}
Let $P$ is a point of type $II$ and $g_t(x,y)=\prod_{j=1}^r\prod_{s=1}^{n_j} f_{j,s}$. The intersection multiplicity  $I(P,f_{a,b},f_{a',b'})$ of any two components $f_{a,b}$ and $f_{a',b'}$  is either $0$ or $p^{i}$.
\end{lemma}
\begin{proof}
Consider $u=f_{a,b}$ and $v=g_m/u$. From Lemma \ref{IntersectionTypeII}  we know that either $I(P,u,v)=0=\sum_{(j,s)\neq (a,b)} I(P,u,f_{j,s})$, then $I(P,f_{a,b},f_{a',b'})=0$ or
$I(P,u,v)=2^i=\sum_{(j,s)\neq (a,b)} I(P,u,f_{j,s})$ using Lemma \ref{OnlyTwohasMultiplicity} we have that there exits $(a',b')$ with $I(P,f_{a,b},f_{a',b'})=p^i$.
\end{proof}

We need some more technical results  for the main theorem,
which give us some upper bounds.

\begin{lemma}\label{Cota} \ \ \ \ \   \ \ \ \ \ \ \  \ \ \ \ \ \
\begin{itemize}
\item[(i)] If $g_t(x,y)$ does not have an absolutely irreducible factor over $\mathbb{F}_p$, then,
\begin{equation}\label{eq100}
\sum_{j=1}^r \deg(f_j)^2/n_j < deg(g_t)^2/2.
\end{equation}
\item[(ii)]
\begin{eqnarray*}
\sum_{j=1}^r \sum_{1\leq i<s\leq n_j} \sum_{P \in Sing(g_t)}  I(P,f_{j,i},f_{j,s})+\sum_{1\leq j<l\leq r} \sum_{\substack{1\leq i\leq n_j\\1\leq s\leq n_l}}
\sum_{\substack{P \in Sing(g_t)\\ P\in II}}  I(P,f_{j,i},f_{l,s})\\\leq p^i ((\ell-1)^2-N_1)
\end{eqnarray*}
\item[(iii)]
\begin{eqnarray*}
\sum_{j=1}^r \sum_{1\leq i<s\leq n_j} \sum_{P \in Sing(g_t)}  I(P,f_{j,i},f_{j,s})+\sum_{1\leq j<l\leq r} \sum_{\substack{1\leq i\leq n_j\\1\leq s\leq n_l}}
\sum_{\substack{ P\in I}}   I(P,f_{j,i},f_{l,s}))\\
\leq (p^{i}+1)(p^i)/2 \quad N_1
\end{eqnarray*}
\item[(iv)]
\begin{eqnarray*}
\sum_{j=1}^r \sum_{1\leq i<s\leq n_j} \sum_{P \in Sing(g_t)}  I(P,f_{j,i},f_{j,s})+\sum_{1\leq j<l\leq r} \sum_{\substack{1\leq i\leq n_j\\1\leq s\leq n_l}}
\sum_{\substack{ P\in III.A}}   I(P,f_{j,i},f_{l,s}))\\
\leq p^{i}(p^i-1)/2 \quad N_2.
\end{eqnarray*}
\end{itemize}
\end{lemma}
\begin{proof}
\begin{itemize}
\item[(i)]
$$\sum_{j=1}^r \deg(f_j)^2/n_j\leq  \sum_{j=1}^r \deg(f_j)^2/2=1/2(deg(f_1)^2+\cdots+deg(f_r)^2)\leq
1/2 deg(g_t)^2$$
\item[(ii)] From Lemma \ref{IntersectionReducibleII} we know that if $P$ is a point of type II then $I(P,f_{j,i},f_{l,s})=0$ for every $j,l\in\{1,\ldots,r\}$ and $1\leq i\leq n_j$,$1\leq s\leq n_l$. From Lemma \ref{IntersectionReducibleII} we now that for each point P of type II there is at most two components $f_{a,b}$ and $f_{a',b'}$ for which $I(P,f_{a,b},f_{a',b'})=p^i$ and zero otherwise. Taking into account that there are $((\ell-1)^2-N_1)$ points of type II we get the result.
\item[(iii)] From Lemma \ref{IntersectionReducibleI} we have that if $P$ is a point of type $I$, then for any two components $f_{a,b}$ and $f_{a',b'}$ we have $I(P,f_{a,b},f_{a',b'})=m_P(f_{a,b})m_P(f_{a',b'})$. Hence we have to prove the following,
\begin{eqnarray*}
\sum_{j=1}^r \sum_{1\leq i<s\leq n_j}   m_P(f_{j,i})m_P(f_{j,s})+
\sum_{1\leq j<l\leq r} \sum_{\substack{1\leq i\leq n_j\\1\leq s\leq n_l}}
 m_P(f_{j,i})m_P(f_{j,s})\\\leq (p^{i}+1)(p^i)/2.
\end{eqnarray*}
    Notice that the left hand side is a maximum when $m_P(f_{j,s})=1$ for every $j\in\{1,\ldots,r\}$ , $s\in \{1,\ldots,n_j\}$.  The latter equation is

\[
\sum_{j=1}^r \sum_{1\leq i<s\leq n_j}   m_P(f_{j,i})m_P(f_{j,s})+
\sum_{1\leq j<l\leq r} \sum_{\substack{1\leq i\leq n_j\\1\leq s\leq n_l}}m_P(f_{j,i})m_P(f_{j,s})
\]
\[
\leq \sum_{j=1}^r \sum_{1\leq i<s\leq n_j} 1+ \sum_{1\leq j<l\leq r} \sum_{\substack{1\leq i\leq n_j\\1\leq s\leq n_l}} 1 \]  
\[  =\binom{p^i+1}{2}=(p^{i}+1)(p^i)/2
\]

\item[(iv)] Same proof as (iii) but taking on account that has $N_2$ singular points of this type with multiplicity $p^i$. 
\end{itemize}
\end{proof}

\begin{lemma}\label{le:BoundDegree}
Let $f_t(x,y)=f_1f_2\ldots f_r=(f_{1,1}\ldots f_{1,n_1})(f_{2,1}\ldots f_{2,n_2})\ldots (f_{r,1}\ldots f_{r,n_r})$
then 
\[
\sum_{j=1}^r \sum_{1\leq i<s\leq n_j} \deg(f_{j,i})\deg(f_{j,s})+
\sum_{1\leq j<l\leq r} \sum_{\substack{1\leq i\leq n_j\\1\leq s\leq n_l}}\deg(f_{j,i})\deg(f_{l,s})\]\[=\frac{1}{2}\biggl(\deg(g_t)^2-\sum_{j=1}^r  \frac{\deg(f_j)^2}{n_j} \biggr)>\deg(g_t)^2/4.
\]
\end{lemma}
\begin{proof}
\begin{equation}\label{bez66}
\sum_{j=1}^r \sum_{1\leq i<s\leq n_j} \deg(f_{j,i})\deg(f_{j,s})+
\sum_{1\leq j<l\leq r} \sum_{\substack{1\leq i\leq n_j\\1\leq s\leq n_l}}\deg(f_{j,i})\deg(f_{l,s}).
\end{equation}
Since each $f_{j,s}$ has the same degree for all $s$,
the first term is equal to
\[
\sum_{j=1}^r  \deg(f_j)^2\ \frac{n_j-1}{2n_j}=
\frac{1}{2} \sum_{j=1}^r  \deg(f_j)^2 - \frac{1}{2} \sum_{j=1}^r  \frac{\deg(f_j)^2}{n_j}.
 \]
 Note that
\begin{eqnarray*}
(\deg (g_t))^2&=&\biggl( \sum_{j=1}^r  \deg (f_j) \biggr)^2\\
&=&  \sum_{j=1}^r \deg (f_j)^2 +2 \biggl( \sum_{1\leq j<l\leq r} \deg (f_j) \deg (f_l) \biggr)\\
&=&   \sum_{j=1}^r \deg (f_j)^2 +2 \sum_{1\leq j<l\leq r} \biggl( \sum_{s=1}^{n_j} \deg (f_{j,s})\biggr)
  \biggl( \sum_{i=1}^{n_l} \deg (f_{l,i})\biggr)\\
  &=&  \sum_{j=1}^r \deg (f_j)^2 +2
  \sum_{1\leq j<l\leq r} \sum_{\substack{1\leq i\leq n_j\\1\leq s\leq n_l}}\deg(f_{j,i})\deg(f_{l,s}).
\end{eqnarray*}
Substituting both of these into (\ref{bez66}) shows that (\ref{bez66}) is equal to
\begin{equation}\label{otherside}
\frac{1}{2}\biggl(\deg(g_t)^2-\sum_{j=1}^r  \frac{\deg(f_j)^2}{n_j} \biggr).
\end{equation}
Using (\ref{eq100}) we get
\begin{equation}\label{othersidecota}
\frac{1}{2}\biggl(\deg(g_t)^2-\sum_{j=1}^r  \frac{\deg(f_j)^2}{n_j} \biggr)>
\frac{1}{2}\biggl(\deg(g_t)^2-\deg(g_t)^2/2 \biggr)=\deg(g_t)^2/4.
\end{equation}

\end{proof}

Finally, here is our main result.

\begin{theo}\label{MainTheo}
In case (B.1)  $g_t(x,y)$ always has an absolutely 
irreducible factor over $\mathbb{F}_p$, if either 
$p\geq 5$, $i\geq 1$ and $\ell> 3$ or
$p\geq 5$, $i\geq 2$ and $\ell\geq 3$ or    
$p=3$, $i\geq 2$ and $\ell\geq 3$.
\end{theo}
\begin{proof}
We apply Bezout's Theorem one more time  to
the product
$$f_1f_2\ldots f_r=(f_{1,1}\ldots f_{1,n_1})(f_{2,1}\ldots f_{2,n_2})\ldots (f_{r,1}\ldots f_{r,n_r}).
$$
The sum of the intersection multiplicities can be written
\[
\sum_{j=1}^r \sum_{1\leq i<s\leq n_j} \sum_{P \in Sing(g_t)}  I(P,f_{j,i},f_{j,s})+
\sum_{1\leq j<l\leq r} \sum_{\substack{1\leq i\leq n_j\\1\leq s\leq n_l}}
\sum_{P \in Sing(g_t)}  I(P,f_{j,i},f_{l,s})
\]
where the first term is for factors within each $f_j$, and the second term
is for cross factors between $f_j$ and $f_l$.
Using Lemma \ref{Cota}, part (ii), (iii) and (iv), the previous sums can be bounded by
\begin{equation}\label{eq17}
\leq (p^{i}+1)(p^i)/2 \  N_1+p^i((l-1)^2-N_1)+p^{i}(p^i-1)/2 \ N_2
\end{equation}

Since $N_1\leq (\ell/3-1)^2$ and $N_2 \leq \ell$ we get

\begin{equation}\label{eq17}
\leq (p^{i}+1)(p^i)/2 \ (\ell/3-1)^2+p^i((l-1)^2-(\ell/3-1)^2)+p^{i}(p^i-1)/2 \ \ell
\end{equation}

On the other hand,
we know from Lemma \ref{le:BoundDegree} that the right-hand side of Bezout's Theorem is bigger than $\deg(g_t)^2/4$.

Hence, so far we have shown that Bezout's Theorem implies the following inequality:
 \[
\deg(g_t)^2/4 \leq (p^{i}+1)(p^i)/2 \ (\ell/3-1)^2+p^i((l-1)^2-(\ell/3-1)^2)+p^{i}(p^i-1)/2 \ \ell
 \]
 Let us now show that the opposite is true, to get a contradiction.  Suppose
 \[
\frac{p^{2i}\ell^2-2p^i\ell+1}{4}> \frac{2(p^{2i}+p^i)}{4} \ (\ell/3-1)^2+\frac{4p^i}{4}((l-1)^2-(\ell/3-1)^2)+\frac{2(p^{2i}-p^i)}{4} \ \ell
 \]
 
 \[
p^{2i}(\ell^2-2(\ell/3-1)^2-2\ell) > p^i(-2(\ell/3-1)^2)+4(\ell-1)^2) -1
\]
 \[
p^{2i}(\frac{7\ell^2}{9}-\frac{2\ell}{3}-2) > p^i(\frac{34\ell^2}{9}-\frac{20\ell}{3}+2) -1
\]
Is enough to see when 
 \[
p^{i} > \frac{\frac{34\ell^2}{9}-\frac{20\ell}{3}+2}{\frac{7\ell^2}{9}-\frac{2\ell}{3}-2}=\frac{\frac{35\ell^2}{9}-\frac{10\ell}{3}-10}{\frac{7\ell^2}{9}-\frac{2\ell}{3}-2}+\frac{\frac{-\ell^2}{9}-\frac{10\ell}{3}+12}{\frac{7\ell^2}{9}-\frac{2\ell}{3}-2}
\]

 \[
p^{i} > 5+\frac{\frac{-\ell^2}{9}-\frac{10\ell}{3}+12}{\frac{7\ell^2}{9}-\frac{2\ell}{3}-2}
\]
Notice that if $\ell=3$ the numerator is positive and if $\ell>3$ then it is negative. Therefore, we need  
$p\geq 5$, $i\geq 1$ and $\ell> 3$ or
$p\geq 5$, $i\geq 2$ and $\ell\geq 3$ or    
$p=3$, $i\geq 2$ and $\ell\geq 3$.

\end{proof}
.

\begin{lemma}\label{le: bound p^i}

If ${\ell}\mid p^i-1$ but ${\ell}\neq p^i-1$,  then
$$
p^i\geq 2(\ell-1)+3.
$$
\end{lemma}
\begin{proof}
Since  ${\ell}\mid p^i-1$ but ${\ell}\neq p^i-1$, then $p^i-1\geq 2\ell\Leftrightarrow p^i\geq 2(\ell-1)+3$.
\end{proof}

Here are some more lemmata we will use.

\begin{lemma}\label{MaxOfMany}
Given $N\in \mathbb{N}$ the values  $x_1,\ldots,x_n$ that maximize the function \\$H(x_1,\ldots,x_n)$ $=\sum_{\substack {1\leq i<j\leq n\\ i\neq j}} x_ix_j$
subject to the constraint
$x_1+\cdots+x_n=N$ are $x_1=\cdots=x_n=N/n$.
\end{lemma}

\begin{lemma}\label{SquareOfDegrees&SquareOfMultiplicities}
If ${\ell}\mid p^i-1$ but ${\ell}\neq p^i-1$,  then
$$
\deg (g_t)^2> \sum_{P\in Sing(g_t)} m_p(g_t)^2.
$$
\end{lemma}
\textbf{Proof:} 
We are going to prove that $\deg (g_t)^2=(p^i\ell-1)^2$ is bigger than  
an upper bound of $\sum_{P\in Sing(g_t)} m_p(g_t)^2$
\begin{eqnarray*}
 \sum_{P\in Sing(g_t)} m_p(g_t)^2
&\leq&
(p^i+1)^2(\ell-1)^2+p^{2i}\ell
\end{eqnarray*}
Notice that all the singular affine points are  of type I. 
We have to prove that ,
$$
(p^{2i}\ell^2-2p^i \ell+1)
>(p^{2i}+2p^i+1)(\ell-1)^2+p^{2i}\ell
$$
Equivalently,
$$
p^{2i}(\ell-1)>p^i(2(\ell-1)^2+2\ell)+(\ell-1)^2-1
$$
Dividing by $p^i({\ell}-1)$ we get
$$
p^{i}>2(\ell-1)+2+\frac{1}{\ell-1}+\frac{(\ell-1)}{p^i}-\frac{1}{p^i(\ell-1)}
$$
We know from Lemma \ref{le: bound p^i} that $p^i\geq 2(\ell-1)+3$ so 
is enough to prove that 
$$
2(\ell-1)+3>2(\ell-1)+2+\frac{1}{\ell-1}+\frac{(\ell-1)}{p^i}-\frac{1}{p^i(\ell-1)},
$$
or
$$
1>\frac{1}{\ell-1}+\frac{(\ell-1)}{p^i}-\frac{1}{p^i(\ell-1)}.
$$
Since $\ell\geq 3$ then 
$\frac{1}{\ell-1}+\frac{(\ell-1)}{p^i}-\frac{1}{p^i(\ell-1)}\leq 1/2+\frac{(\ell-1)}{p^i}-\frac{1}{p^i(\ell-1)}$.

Again since $p^i\geq 2(\ell-1)+3$ we get 
$$1/2+\frac{(\ell-1)}{p^i}-\frac{1}{p^i(\ell-1)}\leq 1-\frac{1}{p^i(\ell-1)}<1.$$  $\square$\bigskip

Let us prove an auxiliary result.
\begin{lemma}\label{Deg<Mult}
All $\mathbb{F}_p$-irreducible components $f_k(x,y)$ of $g_t(x,y)$ satisfy the following conditions:
\begin{itemize}
\item
\begin{equation}\label{eq10}
\deg(f_k)^2\leq \sum_{P \in Sing(g_t)}  m_P(f_k)^2.
\end{equation}
\item
\begin{equation}\label{eq11}
\sum_{1\leq i<j\leq n_k}  m_P(f_{k,i})m_P(f_{k,j})\leq m_P(f_k)^2\frac{n_k-1}{2n_k}.
\end{equation}
\end{itemize}
\end{lemma}
\textbf{Proof:}
Applying Bezout's theorem to $f_k$ gives
\begin{equation}\label{eq12}
\sum_{1\leq i<j\leq n_k} \sum_{P \in Sing(f_k)}
I(P,f_{k,i},f_{k,j})=\sum_{1\leq i<j\leq n_k}   \deg(f_{k,i})
\deg(f_{k,j})=\deg(f_k)^2\frac{n_k-1}{2n_k}.
\end{equation}
Since for every $i,j\in\{1,\ldots,n_k\}$ the tangent cones of $f_{k,i}$ and $f_{k,j}$ consist of
different lines by Lemma \ref{differentfactors},  the left hand side of (\ref{eq12}) is
\begin{equation}\label{eq13}
\sum_{1\leq i<j\leq n_k} \sum_{P \in Sing(f_k)}
I(P,f_{k,i},f_{k,j})= \sum_{P \in Sing(f_k)} \sum_{1\leq i<j\leq n_k}  m_P(f_{k,i})
m_P(f_{k,j})
\end{equation}
because $I(P,u,v)\geq m_P(u) m_P(v)$.
We fix $P$ a singular point.
Applying Lemma \ref{MaxOfMany} to $$\sum_{1\leq i<j\leq n_k}m_P(f_{k,i}) m_P(f_{k,j})$$ subject to
$\sum_{i=1}^{n_k} m_P(f_{k,i})=m_P(f_k)$ we get that
$$
\sum_{1\leq i<j\leq n_k}  m_P(f_{k,i})m_P(f_{k,j})\leq m_P(f_k)^2\frac{n_k-1}{2n_k}
$$
which proves (\ref{eq11}).
Summing over $P$ then proves (\ref{eq10}).
$\square$\bigskip

\begin{theo}\label{MainTheo1}
If  $gcd(\ell,p^i-1)=\ell$ (case B.2) and $\ell<p^i-1$, $g_t(x,y)$ always has an absolutely irreducible factor over $\mathbb{F}_p$.
\end{theo}

We apply Bezout's Theorem  one more time  to
the product
$$f_1f_2\ldots f_r=(f_{1,1}\ldots f_{1,n_1})(f_{2,1}\ldots f_{2,n_2})\ldots (f_{r,1}\ldots f_{r,n_r}).
$$
The sum of the intersection multiplicities can be written
\[
\sum_{k=1}^r \sum_{1\leq i<j\leq n_k} \sum_{P \in Sing(g_t)}  I(P,f_{k,i},f_{k,j})+
\sum_{1\leq k<l\leq r} \sum_{\substack{1\leq i\leq n_k\\1\leq j\leq n_l}}
\sum_{P \in Sing(g_t)}  I(P,f_{k,i},f_{l,j})
\]
where the first term is for factors within each $f_k$, and the second term
is for cross factors between $f_k$ and $f_l$.
Since for every $k$ and $i$
the tangent cones of the $f_{k,i}$ consist of different lines by Lemma \ref{differentfactors},
the previous sums can be written
\begin{equation}\label{eq17}
\sum_{P \in Sing(g_t)}\biggl[
\sum_{k=1}^r \sum_{1\leq i<j\leq n_k}  m_P(f_{k,i})m_P(f_{k,j})+
\sum_{1\leq k<l\leq r} \sum_{\substack{1\leq i\leq n_k\\1\leq i\leq n_l}}   m_P(f_{k,i})m_P(f_{l,j})\biggr].
\end{equation}

Note that
\begin{eqnarray*}
(m_P (g_t))^2&=&\biggl( \sum_{k=1}^r  m_P (f_k) \biggr)^2\\
&=&  \sum_{k=1}^r m_P (f_k)^2 +2 \biggl( \sum_{1\leq k<l\leq r} m_P (f_k) m_P (f_l) \biggr)\\
&=&   \sum_{k=1}^r m_P (f_k)^2 +2 \sum_{1\leq k<l\leq r} \biggl( \sum_{i=1}^{n_k} m_P (f_{k,i})\biggr)
  \biggl( \sum_{j=1}^{n_l} m_P (f_{l,j})\biggr)\\
  &=&  \sum_{k=1}^r m_P (f_k)^2 +2
  \sum_{1\leq k<l\leq r} \sum_{\substack{1\leq i\leq n_k\\1\leq j\leq n_l}}m_P(f_{k,i})m_P(f_{l,j}).
\end{eqnarray*}
Substituting, (\ref{eq17}) becomes
\begin{equation}\label{eq77}
\sum_{P \in Sing(g_t)}\biggl[
\sum_{k=1}^r \sum_{1\leq i<j\leq n_k}  m_P(f_{k,i})m_P(f_{k,j})+
\frac{1}{2}\biggl( m_P(g_t)^2- \sum_{k=1}^r m_P (f_k)^2 \biggr)\biggr].
\end{equation}
Substituting (\ref{eq11}) this is
\begin{eqnarray}\label{eq18}
&\leq &\sum_{P \in\ Sing(g_t)}\biggl[
\sum_{k=1}^r m_P(f_k)^2\frac{n_k-1}{2n_k} +
 \frac{1}{2}\biggl( m_P(g_t)^2- \sum_{k=1}^r m_P (f_k)^2 \biggr)\biggr]\\
&=&\frac{1}{2} \sum_{P \in Sing(g_t)}\biggl[
  m_P(g_t)^2 -  \sum_{k=1}^r  \frac{m_P(f_k)^2}{n_k}\label{bnm}\biggr].
\end{eqnarray}

On the other hand,
from Lemma \ref{le:BoundDegree} we know that the right-hand side of Bezout's Theorem 
\begin{equation}\label{bez6}
\sum_{k=1}^r \sum_{1\leq i<j\leq n_k} \deg(f_{k,i})\deg(f_{k,j})+
\sum_{1\leq k<l\leq r} \sum_{\substack{1\leq i\leq n_k\\1\leq j\leq n_l}}\deg(f_{k,i})\deg(f_{l,j}).
\end{equation}
is equal to
\begin{equation}\label{otherside}
\frac{1}{2}\biggl(\deg(g_t)^2-\sum_{k=1}^r  \frac{\deg(f_k)^2}{n_k} \biggr).
\end{equation}
 Comparing (\ref{otherside}) and (\ref{bnm}),
 so far we have shown that Bezout's Theorem implies the following inequality:
 \[
 \deg(g_t)^2-\sum_{k=1}^r  \frac{\deg(f_k)^2}{n_k} \leq
 \sum_{P \in Sing(g_t)}\biggl[
  m_P(g_t)^2 -  \sum_{k=1}^r  \frac{m_P(f_k)^2}{n_k}\biggr].
 \]
 Finally, using (\ref{eq10})
and  Lemma \ref{SquareOfDegrees&SquareOfMultiplicities} to compare both sides
term by term,
this is a contradiction.
$\square$\bigskip

\begin{remark}
$t=\frac{3^k+1}{2}$ is an exceptional number over $\mathbb{F}_{3^m}$ whenever 
$p\nmid k$ and $(m.k)=1$. The latter result it is not a contradiction because 
we are ssuming $t=p^i(\ell)+1$.

Notice that then  $\frac{3^k+1}{2}=p^i(\ell)+1\Leftrightarrow 3^k=p^i(2\ell)+1$ which is imposible since th right hand side it is not divisible by $3$. 
\end{remark}

\begin{remark}
Note that this proof fails if $\ell = p^i-1$. This case remains still unproven. 
\end{remark}

\section{Results on Case (A)}\label{caseAresults}

Throughout  this section we shall assume that 
we are in case (A), i.e., $p$ does not divide either $t$ nor $t-1$.

The proofs in this section use the theory of pencils and Bertini's theorem.
The background was given in section \ref{backbertini}.
We also use some of the previous methods, singularities and Bezout's theorem.



Let us consider the projective plane $\mathbb{P}^2$ over $\overline{\mathbb{F}}_p$ and take homogeneous coordinates $(X:Y:Z)$ such that $x:=X/Z$ and $y:=Y/Z$ are affine coordinates in the chart defined by $Z\not=0$. Let $F_t(X,Y,Z)$ (resp., $G_t(X,Y,Z)$) be the homogenization of the polynomial $f_t(x,y)$ (resp., $g_t(x,y)$) and denote by 
 $\chi_t$ the projective curve over $\overline{\mathbb{F}}_p$ defined by the equation $G_t(X,Y,Z)=0$. Notice that $g_t(x,y)$ has an absolutely irreducible factor over $\mathbb{F}_p$ if and only if $G_t(X,Y,Z)$ does so. For any subfield 
 $K\subseteq \overline{\mathbb{F}}_p$, $\chi_t(K)$ will denote the set of $K$-rational points of $\chi_t$. 
 
By Lemma \ref{le: Mult2casoA} one has that, in case (A), the singularities of $\chi_t$ are exactly those singular points of the curve $F_t(X,Y,Z)=0$ which do not belong to the line $X=Y$; moreover all of them are in the chart $Z\not=0$. As a consequence
 the singular locus of $\chi_t$ is defined (in the affine coordinates $x,y$) by the condition $x\neq y$ and the equations
\begin{equation}\label{eq:SINx}
(x+1)^{t-1}-x^{t-1}=0,
\end{equation}
\begin{equation}\label{eq:SINy}
(y+1)^{t-1}-y^{t-1}=0\;\;\;\mbox{and}
\end{equation}
\begin{equation}\label{eq:SINxy}
x^{t-1}-y^{t-1}=0
\end{equation}
Taking into account that all the coordinates of a singular point must be non-zero, set
$$r:=\frac{x+1}{x},\;\;\; s:=\frac{y+1}{y}.$$
Conditions (\ref{eq:SINx}) and (\ref{eq:SINy}) mean that 
\begin{equation}\label{ee4}
r^{t-1}=1,\;\;\;\; s^{t-1}=1.
\end{equation}
Condition (\ref{eq:SINxy}) means that
\begin{equation}\label{ee5}
(r-1)^{t-1}=(s-1)^{t-1}.
\end{equation}
From these considerations it is straightforward that the map defined by $$(r,s)\mapsto \left(\frac{1}{r-1}:\frac{1}{s-1}:1\right)$$ provides a bijection between the set $$\Omega_t:=\left\{(r,s)\in \overline{\mathbb{F}}_p^2 \mid  r^{t-1}=s^{t-1}=1, (r-1)^{t-1}=(s-1)^{t-1}, r\not=1, s\not=1, r\not=s \right\}$$ and the set of singular points of $\chi_t$. Moreover, by Lemma \ref{le:Proct2LinesCaseA}, all these singularities are nodal.
The following lemma will be a key tool for our results.

\begin{lemma}\label{le:exit2rationalPoints}
Let $K$ be a finite subfield of $\overline{\mathbb{F}}_p$. If there exists $P\in \chi_t(K)$  such that $P$ is not a singular point of $\chi_t$ then there exists an absolutely irreducible homogeneous polynomial $H(X,Y,Z)\in K[X,Y,Z]$ such that $H(X,Y,Z)$ divides $G_t(X,Y,Z)$.
\end{lemma}

\begin{proof}
Let $G_t(X,Y,Z)=R_1(X,Y,Z)\cdots R_m(X,Y,Z)$ be the decomposition of 
$G_t(X,Y,Z)$ as a product of irreducible homogeneous polynomials $R_i\in K[X,Y,Z]$, $1\leq i\leq m$. Assume, without loss of generality, that $R_1(P)=0$. By Lemma \ref{IfIrreducibleEqualDegreeFactors} there exists a finite extension $L$ of $K$ and an absolutely irreducible homogeneous polynomial $H(X,Y,Z)\in L[X,Y,Z]$ such that $$R_1(X,Y,Z)=c\prod_{\sigma\in G} \sigma(H(X,Y,Z)),$$
where $c\in L$ and $G=Gal(L/K)$. The point $P$ must be a zero of any of the above factors $\sigma(H(X,Y,Z))$. But, since $P$ is a non-singular point of $\chi_t$, it holds that $L=K$ and, therefore, $R_1(X,Y,Z)$ is absolutely irreducible.
\end{proof}

\begin{theo}\label{gordo1}
If either $\gcd(p-1,t)\geq 3$ or $\gcd(p-1,t-1)\geq 2$ then $G_t(X,Y,Z)$ has an absolutely irreducible factor over $\mathbb{F}_p$.
\end{theo}
\begin{proof}
Assume first that $\gcd(p-1,t)\geq 3$. Then the equation $z^t-1=0$ has, at least, 3 solutions in $\mathbb{F}_p$. This implies the existence of $r_1,r_2\in \mathbb{F}_p\setminus\{1\}$ such that $r_1\not= r_2$ and $r_1^t=r_2^t=1$. It is straightforward that  $P=(1/(r_1-1):1/(r_2-1):1)$ is an $\mathbb{F}_p$-rational point of the curve $\chi_t$. Moreover it is clearly non-singular (see the paragraph before Lemma \ref{le:exit2rationalPoints}). Therefore $G_t(X,Y,Z)$ has an absolutely irreducible factor over $\mathbb{F}_p$ by Lemma \ref{le:exit2rationalPoints}.

Assume now that $\gcd(p-1,t-1)\geq 2$. Then the equation $z^{t-1}-1=0$ has, at least, $2$ solutions in $\mathbb{F}_p$, say $1$ and $r\not=1$. It is straightforward that the $\mathbb{F}_p$-rational point $(r:1:0)$ belongs to $\chi_t$; moreover it is non-singular because it is on the line at infinity $Z=0$. Applying again Lemma \ref{le:exit2rationalPoints} one has that $G_t(X,Y,Z)$ has an absolutely irreducible factor over $\mathbb{F}_p$.
\end{proof}


A natural generalization is given in the following result:

\begin{theo}
If there exist $m\in\mathbb{N}$ such that $G_t(X,Y,Z)$ does not factor over $\mathbb{F}_{p^m}$ and furthermore either $gcd(p^m-1,t)\geq 3$ or $gcd(p^m-1,t-1)\geq 2$ then $G_t(X,Y,Z)$ has an absolutely irreducible factor over $\mathbb{F}_p$.
\end{theo}

The next result shows an upper bound for the number of singular points of $\chi_t$:
\begin{prop}\label{joselito}
If $t$ is even then the number of singular points of $\chi_t$ is, at most, $\frac{(t-2)(t-4)}{2}$.
\end{prop}

\begin{proof}
Let us define $$\Omega'_t:=\left\{(r,s)\in \overline{\mathbb{F}}_p^2 \mid  r^{t-1}=s^{t-1}=1, r\not=1, s\not=1, (r-1)^{t-1}=(s-1)^{t-1} \right\}$$ and $\Delta:=\{(r,r)\in \overline{\mathbb{F}}_p^2 \mid  r^{t-1}=1, r\not=1\}$. Observe that $\Delta\subseteq \Omega'_t$ and $\Omega_t=\Omega'_t\setminus \Delta$.

From the sequence of equalities $(1/s-1)^{t-1}=-(s-1)^{t-1}/s^{t-1}=-(s-1)^{t-1}=-(r-1)^{t-1}$, it follows the implication
$$(r,s)\in \Omega'_t\Rightarrow (r,1/s)\not\in \Omega'_t.$$
Then $\#\Omega'_t\leq \frac{t-2}{2}(t-2)$ and therefore
$$\# \Omega_t=\#\Omega'_t-\#\Delta\leq \frac{(t-2)^2}{2}-(t-2)=\frac{(t-2)(t-4)}{2}.$$
The result holds because there is a bijection between $\Omega_t$ and the set of singular points of $\chi_t$.

\end{proof}

\begin{theo}\label{te:MainTheoGeneralCase}
If the number of singular points of $\chi_t$ is less than $\frac{(t-2)^2}{4}$ then $g_t(x,y)$ has an absolutely irreducible factor over $\mathbb{F}_p$.
\end{theo}
\begin{proof}
As in the proof of Theorem \ref{MainTheo}, let $g_t(x,y)=\prod_{j=1}^r f_j$ be the decomposition of $g_t(x,y)$ as a product of irreducible polynomials in $\mathbb{F}_p[x,y]$. Let also $f_j=\prod_{i=1}^{n_j}f_{j,i}$ be the decomposition of $f_j$ as a product of absolutely irreducible polynomials, $1\leq j\leq r$. Applying Bezout's Theorem we have that
$$\sum_{(j_1,k_1)\not=(j_2,k_2)} \sum_{P} I(P,f_{j_1,k_1},f_{j_2,k_2})=\sum_{(j_1,k_1)\not=(j_2,k_2)}  \deg(f_{j_1,k_1})\deg(f_{j_2,k_2}),$$
where $P$ varies over the set of singular points of the affine curve $g_t(x,y)=0$ (notice these are all the singular points of $\chi_t$). Since all the singularities are nodal, each intersection number $I(P,f_{j_1,k_1},f_{j_2,k_2})$ is equal to 1 (resp., $0$) if $f_{j_1,k_1}(P)=f_{j_2,k_2}(P)=0$ (resp., otherwise). Therefore we have the inequality
$$N_t\geq \sum_{(j_1,k_1)\not=(j_2,k_2)}  \deg(f_{j_1,k_1})\deg(f_{j_2,k_2}),$$
where $N_t$ denotes the number of singular points of $\chi_t$. Taking into account that, for each $j=1,\ldots,r$, all the absolutely irreducible components of $f_j$ have the same degree, we know from Lemma \ref{le:BoundDegree} that the right hand side of the above inequality is

\[
\sum_{j=1}^r \sum_{1\leq i<s\leq n_j} \deg(f_{j,i})\deg(f_{j,s})+
\sum_{1\leq j<l\leq r} \sum_{\substack{1\leq i\leq n_j\\1\leq s\leq n_l}}\deg(f_{j,i})\deg(f_{l,s})\]\[=\frac{1}{2}\biggl(\deg(g_t)^2-\sum_{j=1}^r  \frac{\deg(f_j)^2}{n_j} \biggr)>\deg(g_t)^2/4=\frac{(t-2)^2}{4}\]
and this is a contradiction with our assumption $N_t< (t-2)^2/4$.

\end{proof}

Denote by $\mu(t-1)$ the set of $(t-1)$-roots of unity in $\overline{\mathbb{F}}_p$. We define the following equivalence relation in $\mu(t-1)\setminus \{1\}$: $$r\; {\mathcal R}_t \;s \mbox{ if }  (r-1)^{t-1}=(s-1)^{t-1}.$$

\begin{lemma}\label{torero}
Let $\{ R_1,\ldots,R_{\ell}\}$ be the quotient set of ${\mathcal R}_t$. Assume that $\ell\geq 4$ and, for each $j_1\in \{1,2,\ldots,\ell\}$, there exists a subset $\{j_1,j_2,j_3,j_4\}\subseteq \{1,2,\ldots,\ell\}$ of cardinality $4$ such that the equivalence classes $R_{j_1},R_{j_2}, R_{j_3}$ and $R_{j_4}$ have the same cardinality. Then $g_t(x,y)$ has an absolutely irreducible factor over $\mathbb{F}_p$.
\end{lemma}
\begin{proof}
For each $m\in \{\#R_1,\#R_2,\ldots,\# R_{\ell} \}$ let $A_m$ be the union of all those classes $R_i$ with cardinality $m$. Consider the partition of $\mu(t-1)\setminus \{1\}$ given by ${\mathcal A}:=\{A_m\mid \#R_i=m \mbox{ for some } i\}$. Notice that, by hypothesis, each $A\in {\mathcal A}$ is the union of, at least, 4 equivalence classes $R_i$.  Then the
number of singular points of the curve $\chi_t$ admits the following expression:
$$\# \Omega_t=\sum_{A\in {\mathcal A}} c(A) n(A)^2 -(t-2),$$
where $c(A)\geq 4$ denotes the number of equivalence classes $R_i$ whose union is $A$ and $n(A)$ denotes the cardinality of any of these classes. Notice that $\sum_{A\in{\mathcal A}} c(A)=\ell$ and $\sum_{A\in {\mathcal A}} c(A) n(A)=t-2$. Using repeatedly the inequality $(x+y)^2\geq x^2+y^2$ for $x,y\in \mathbb{R}$, it is not hard to prove that the sum $\sum_{A\in {\mathcal A}} c(A) n(A)^2$ achieves (for fixed $t$) the maximum value when there is only one summand and with coefficient 4. Therefore:
$$\#\Omega_t\leq 4\left(\frac{t-2}{4}\right)^2-(t-2)<\frac{(t-2)^2}{4}.$$
Now the result follows from Theorem \ref{te:MainTheoGeneralCase}.

\end{proof}

\begin{remark}
Assume that $t=\frac{p^\ell+1}{2}$ for some $\ell\geq 1$ and that $t$ is even. Then $2(t-1)=p^\ell-1$ and, therefore, the set of $2(t-1)$-roots of unity of $\overline{\mathbb{F}}_p$ is the multiplicative group $\mathbb{F}^*_{p^\ell}$. Hence $(r-1)^{2(t-1)}=1$ for any $r\in \mu(t-1)\setminus \{1\}$ because $r-1\in \mathbb{F}^*_{p^\ell}$. Since $(r-1)^{t-1}=-(1/r-1)^{t-1}$ this means that $(r-1)^{t-1}=1$ (resp., $(r-1)^{t-1}=-1$) for $\frac{t-2}{2}$ elements $r$ of $\mu(t-1)\setminus \{1\}$. So, the relation ${\mathcal R}_t$ has exactly two equivalence classes; then Lemma \ref{torero} cannot be applied to this case. Further, the number of singular points of  $\chi_t$ is
$$\#\Omega_t=2\left(\frac{t-2}{2}\right)^2-(t-2)=\frac{(t-2)(t-4)}{2},$$
which is the maximum number of possible singular points by Proposition \ref{joselito}. Then the existence of an absolutely irreducible (over $\mathbb{F}_p$) factor of $g_t(x,y)$ cannot be proved, for these values of $t$, by bounding the number of singular points (that is, applying Theorem \ref{te:MainTheoGeneralCase}).

\end{remark}

\begin{theo}\label{banderillero}
If $t$ is even, $(s-1)^{t-1}\not\in \mathbb{F}_{p}$ and $(s-1)^{(t-1)(p-1)}\not=-1$ for all $s\in \mu(t-1)\setminus \{1\}$ then then $g_t(x,y)$ has an absolutely irreducible factor over $\mathbb{F}_p$.
\end{theo}

\begin{proof}
Fix $r\in \mu(t-1)\setminus \{1\}$ and set $b:=(r-1)^{t-1}$. Denote by $R(r)$ the equivalence class of $r$ with respect to the relation $\mathcal R_t$ (analogously with $R(1/r)$, $R(r^p)$ and $R(1/r^p)$). These 4 classes are distinct because the elements $b$,  $(1/r-1)^{t-1}=-b$, $(r^p-1)^{t-1}=b^p$ and $(1/r^p-1)^{t-1}=-b^p$ are distinct due to the hypothesis in the statement. The map $\varphi_1: R(r)\rightarrow R(1/r)$ (resp., $\varphi_2: R(r)\rightarrow R(r^p)$) (resp., $\varphi_3: R(1/r)\rightarrow R(1/r^p)$) defined by $s\mapsto 1/s$ (resp., $s\mapsto s^p$) (resp., $s\mapsto s^p$) is  injective.  $\varphi_1$ is clearly surjective. Conjugation provides injective maps $R(r^{p^{e-1}})\rightarrow R(r^{p^{e}})$ for $e\geq 1$; taking $e$ as the cardinality of the conjugacy class of $r$ one has an injective map $R(r^{p^{e-1}})\rightarrow R(r)$, and this implies that the map $\varphi_2$ is a bijection. Analogously $\varphi_3$ is bijective.
 Therefore $R(r)$,  $R(1/r)$, $R(r^p)$ and  $R(1/r^p)$ are 4 distinct equivalence classes of the relation $\mathcal R_t$ with the same cardinality. Since all these facts are valid for an arbitrary element $r\in \mu(t-1)\setminus \{1\}$, applying Lemma \ref{torero} we have that $g_t(x,y)$ has an absolutely irreducible factor over $\mathbb{F}_p$.

\end{proof}

\begin{remark}\label{josetomas}
Notice that the following condition implies the hypothesis of Theorem \ref{banderillero}:
 $(s-1)^{t-1}\not\in \mathbb{F}_{p^2}$ for all $s\in \mu(t-1)\setminus \{1\}$. 
\end{remark}


\begin{coro}\label{newprop}
Assume that $t$ is even. Then $g_t(x,y)$ has an absolutely irreducible factor over $\mathbb{F}_p$ if the following equivalent conditions are satisfied:
\begin{itemize}
\item[(a)] $t-1$ divides $p^{2e}+1$ for some positive integer $e$.
\item[(b)] The order $u$ of $p$ in $\mathbb{Z}/(t-1)\mathbb{Z}$ is a multiple of 4 and $t-1$ divides $p^{u/2}+1$.
\end{itemize}

\end{coro}

\begin{proof}
Let us prove first that (a) implies that $g_t(x,y)$ has an absolutely irreducible factor over $\mathbb{F}_p$. If $(s-1)^{t-1}\not\in \mathbb{F}_{p^2}$ for all $s\in \mu(t-1)\setminus \{1\}$ then, 
by Theorem \ref{banderillero} and Remark \ref{josetomas}, $g_t(x,y)$ has an absolutely irreducible factor over $\mathbb{F}_p$. Otherwise, there exists $r\in \mu(t-1)\setminus \{1\}$ such that $(r-1)^{t-1}=b\in \mathbb{F}_{p^2}$. Consider the polynomial $P(x)=(x-1)^{t-1}-b\in \mathbb{F}_{p^2}$. Notice that $r^{p^{2e}}=1/r$  because, by assumption, $t-1$ divides $p^{2e}+1$. So, $1/r$ is a conjugate of $r$ over $\mathbb{F}_{p^2}$ and, then, $1/r$ must be a root of $P(x)$. But this is a contradiction because $P(1/r)=-2b\not=0$.

Finally we will prove that (a) implies (b) (the converse implication is trivial). So, assume that $t-1$ divides $p^{2e}+1$ for some positive integer $e$. A clear consequence of this is that $4e$ is a multiple of $u$. 

We claim that $4e=uk$, where $k$ is odd (and, in particular, $4$ divides $u$). Indeed, reasoning by contradiction, if $k$ is even then $2e$ is a multiple of $u$ and, therefore, $p^{2e}\equiv 1\;({\rm mod}\; t-1)$, a contradiction. Then $4e=u(2m+1)$ for some $m\in \mathbb{N}$. Therefore
$$-1\equiv p^{2e}\equiv \frac{u}{2}(2m+1)\equiv p^{um} p^{u/2}\equiv p^{u/2}\;({\rm mod}\; t-1).$$

\end{proof}

\begin{coro}

If $t-1\geq 3$ is a prime number such that the multiplicative order of $p$ in  $\mathbb{Z}/(t-1)\mathbb{Z}$ is $(t-2)/2$ then $g_t(x,y)$ has an absolutely irreducible factor over $\mathbb{F}_p$.

\end{coro}

\begin{proof}
On the one hand, since $t-1$ is prime and, by \cite[Ex. 3.36]{LN1}, the $(t-1)$-th cyclotomic polynomial $Q_{t-1}$ is irreducible over $\mathbb{F}_{p^2}$, we have that all the elements in $\mu(t-1)\setminus \{1\}$ are conjugate over $\mathbb{F}_{p^2}$. On the other hand, by Theorem \ref{banderillero} and Remark \ref{josetomas}, we can assume that there exists $r\in \mu(t-1)\setminus \{1\}$ such that $b:=(r-1)^{t-1}\in \mathbb{F}_{p^2}$.
These two facts imply that the $t-2$ elements of $\mu(t-1)\setminus \{1\}$ are roots of the polynomial $(x-1)^{t-1}-b\in \mathbb{F}_{p^2}[x]$ and, therefore, the equivalence relation ${\mathcal R}_t$ has only one equivalence class. Then the number of singular points of $\chi_t$, that is, $\# \Omega_t$, is $(t-2)^2-(t-2)=(t-2)(t-3)$. This is a contradiction with Proposition \ref{joselito} because this number is strictly greater than $(t-2)(t-4)/2$.

\end{proof}


For each natural number $n$ denote by $K^{(n)}$ the cyclotomic extension of $\mathbb{F}_p$ given by the splitting field of the polynomial $x^n-1\in \mathbb{F}_p [x]$.

\begin{lemma}\label{lemagordo}
 Set $g_t(x,t)=h_1(x,y)h_2(x,y)\cdots h_{\ell}(x,y)$, where $h_i(x,y)\in \overline{\mathbb{F}}_p$ is an absolutely irreducible polynomial for each $i\in\{1,2,\ldots,\ell\}$, and consider the field $K_t:=K^{(t)}\cap K^{(t-1)}$. 

\begin{itemize}
\item[(a)] For each $i\in\{1,2,\ldots,\ell\}$, it holds that $h_i(x,y)=\gamma_i h'_i(x,y)$, where $h'_i(x,y)\in K_t[x,y]$ and $\gamma_i\in \overline{\mathbb{F}}_p$.

\item[(b)] For each positive divisor $d$ of $t$ such that $d>2$, there exists $i\in\{1,2,\ldots,\ell\}$ such that $h_i(x,y)=\alpha_1 h_{i,1}(x,y)$, where $h_{i,1}(x,y)\in K^{(d)}[x,y]$ and $\alpha_1\in \overline{\mathbb{F}}_p$.

\item[(c)] For each divisor $d>1$ of $t-1$ there exists $i\in\{1,2,\ldots,\ell\}$ such that $h_i(x,y)=\alpha_2 h_{i,2}(x,y)$, where $h_{i,2}(x,y)\in K^{(d)}[x,y]$ and $\alpha_2\in \overline{\mathbb{F}}_p$.
\end{itemize}
\end{lemma}

\begin{proof}

Let us consider the pencil $\cp(P,Q)$, where $$P(X,Y,Z):=\frac{(X+Z)^t-X^t}{Z}\;\;\;\mbox{ and }\;\;\; Q(X,Y,Z):=\frac{(Y+Z)^t-Y^t}{Z}.$$
Notice that $P-Q=F_t$, where $F_t(X,Y,Z)$ is the homogeneization of the (affine) polynomial $f_t(x,y)$.  It is straightforward to show that the cluster of base points of $\cp(P,Q)$  is
$${\mathcal C}_1:=\left\{ \left(\frac{1}{r-1}:\frac{1}{s-1}:1\right) \mid r^t=s^t=1 \mbox{ and } r,s\not=1\right\}\subseteq \gp^2.$$
The curve defined by $P(X,Y,Z)=0$ (resp., $Q(X,Y,Z)=0$) is a union of lines which are transversal to those defined by $P(X,Y,Z)=0$ (resp., $Q(X,Y,Z)=0$). Therefore, the multiplicity of a general curve of the pencil at each one of the base points is $1$ and, by Corollary \ref{bbb}, $\cp(P,Q)$ is an irreducible pencil. 

Let $h_i(x,y)$ be any of the (absolutely) irreducible components of $g_t(x,y)$ and denote by $H_i(X,Y,Z)$ its homogenization. Let $d_i$ be the degree of $H_i$ and let $\nu:{\mathcal C}_1\rightarrow \mathbb{N}$ be the map such that, for each points $p\in {\mathcal C}_1$, $\nu(p)$ is equal to 1 if $H_i$ vanishes at $p$ and $0$ otherwise. 

We claim that the vector space ${\mathcal L}_{d_i}({\mathcal C}_1,{\bold{\nu}})$ is spanned by $H_i$. Indeed, it is obvious that $H_i\in {\mathcal L}_{d_i}({\mathcal C},\bold{\nu})$ and we shall reason by contradiction assuming that $\dim_{\overline{\mathbb{F}}_p} {\mathcal L}_{d_i}({\mathcal C}_1,\bold{\nu})\geq 2$. In this case, there exists $T(X,Y,Z)\in {\mathcal L}_{d_i}(\Gamma_1,\bold{\nu})$ such that $\{H_i,T\}$ is linearly independent over $\overline{\mathbb{F}}_p$. This contradicts the irreducibility of the pencil $\cp(P,Q)$ because $[\alpha H_i+\beta T](X-Y)\prod_{j\not=i} H_j$ belongs to  ${\mathcal L}_d({\mathcal C}_1,{\bold m})$ for any $(\alpha,\beta)\in \mathbb{P}^1$ (where $\bold{m}(p):=1$ for each base point $p$) and this space is equal to ${\mathcal P}(P,Q)$ by Lemma \ref{ccc}. Now, applying Lemma \ref{eee} to ${\mathcal L}_{d_i}({\mathcal C}_1,\bold{\nu})$, one has that $h_i(x,y)=\gamma_i' h'_i(x,y)$, where $h'_i(x,y)\in K^{(t)}[x,y]$ and $\gamma_i'\in \overline{\mathbb{F}}_p$.

Let us consider now the pencil ${\mathcal P}(F_t,Z^{t-1})$. The set of its base points is $\{s_r:=(r,1,0)\mid r^{t-1}=1\}$ and, localizing a general member of the pencil at any of the points $s_r$ and analyzing the evolution by blow-ups, it is easy to deduce that the cluster of base points of ${\mathcal P}(F_t,Z^t)$ is
${\mathcal C}_2:=\cup_r {\mathcal S}_r$,
where $r$ varies in the set of $(t-1)$-roots of unity of $\algclosure$ and 
$${\mathcal S}_r:=\{s_{r,1}:=s_r,s_{r,2},\ldots, s_{r,t-1}\},$$
$s_{r,i}$ being the intersection point of the exceptional divisor of the blow-up centered at $s_{r,i-1}$ and the strict transform of the line at infinity $Z=0$, $2\leq i\leq t-1$. Therefore any point of ${\mathcal C}_2$ is $K^{(t-1)}$-rational. Using a similar reasoning as above (but now considering the strict transforms at the points of ${\mathcal C}_2$ of the curves $H_i=0$ in the definition of the map $\nu$) it holds that, for any $i\in \{1,2,\ldots,\ell\}$, $h_i(x,y)=\gamma_i'' h''_i(x,y)$, where $h''_i(x,y)\in K^{(t-1)}[x,y]$ and $\gamma_i''\in \overline{\mathbb{F}}_p$.

Hence, we have deduced that, for each $i\in \{1,2,\ldots,\ell\}$, $$h_i(x,y)=\gamma_i'h'_i(x,y)=\gamma_i''h''_i(x,y),$$ where $h'_i(x,y)\in K^{(t)}[x,y]$ and $h''_i(x,y)\in K^{(t-1)}[x,y]$  and $\gamma_i',\gamma_i''\in \overline{\mathbb{F}}_p$. If $\alpha\in K^{(t)}$ is one of the coefficients of $h_i'$ it is clear from the above equalities that $\beta:=\frac{\gamma_i'}{\gamma_i''}\alpha\in K^{(t-1)}$. Taking $\hat{h}_{i}(x,y):=\frac{1}{\alpha} h'_i(x,y)\in K^{(t)}[x,y]$ one has that $\hat{h}_{i}(x,y)=\frac{1}{\beta} h''(x,y)\in K^{(t-1)}[x,y]$ and, therefore, $\hat{h}_i(x,y)\in K_t[x,y]$. Part (a) follows by observing that $h_i(x,y)=\gamma_i'\alpha \hat{h}_{i}(x,y)$.

To prove (b), observe that ${\mathcal C}_1$ contains non-singular $K^{(d)}$-rational points of the curve $\chi_t$ for any cyclotomic field $K^{(d)}$ with $d$ dividing $t$ and $d>2$ (specifically the points $(1/(r-1):1/(s-1):1)$, $r\not=1$ and $s\not=1$ being two \emph{distinct} $d$-roots of unity). Now the result follows by Lemma \ref{le:exit2rationalPoints}.

The proof of (c) is similar taking into account that ${\mathcal C}_2$ contains non-singular $K^{(d)}$-rational points of the curve $\chi_t$ for any cyclotomic field $K^{(d)}$ with $d$ dividing $t-1$ and $d>1$ (specifically the points $(r:1:0)$, $r\not=1$ being a $d$-root of unity).

\end{proof}

In the next theorem, for every positive integer $n$ that is not divisible by $p$, we shall denote by $e_n$ the degree of the cyclotomic extension $K^{(n)}/\mathbb{F}_p$, that is, the multiplicative order of $p$ in $\mathbb{Z}/n\mathbb{Z}$.

\begin{theo}\label{spiderman}
Set $e:=\gcd(e_{t-1},e_t)$ and assume that at least one of these conditions holds:
\begin{itemize}
\item[(1)] There exists a divisor $d>2$ of $t$ such that $\gcd(e,e_d)=1$.
\item[(2)] There exists a divisor $d>1$ of $t-1$ such that $\gcd(e,e_d)=1$.
\end{itemize}
Then the polynomial $g_t(x,y)$ has an absolutely irreducible factor in $\mathbb{F}_p[x,y]$. \end{theo}

\begin{proof}
Let $K_t$ be as in Lemma \ref{lemagordo}. Then $e=|K_t:\mathbb{F}_p|$ and, therefore, $K_t=\mathbb{F}_{p^e}$. Suppose that there exists a divisor $d>2$ of $t$ such that $\gcd(e,e_d)=1$. This implies that $K^{(d)}\cap K_t=\mathbb{F}_{p^{e_d}}\cap \mathbb{F}_{p^e}=\mathbb{F}_p$. On the one hand, by Part (b) of Lemma \ref{lemagordo}, there exists an absolutely irreducible factor $h(x,y)$ of $g_t(x,y)$ such that $h(x,y)\in K^{(d)}[x,y]$. On the other hand, by Part (a) of Lemma \ref{lemagordo}, $\beta h(x,y)\in K_t[x,y]$ for some $\beta\in \algclosure$. Take a coefficient $\alpha$ of $h(x,y)$ and define $\hat{h}(x,y):=\frac{1}{\alpha}h(x,y)\in K^{(d)}[x,y]$. Then $\hat{h}(x,y)=\frac{1}{\alpha \beta} \beta h(x,y)\in K_t[x,y]$ because $\beta \alpha\in K_t$ and, therefore, $\hat{h}(x,y)\in \mathbb{F}_p[x,y]$. Hence, $\hat{h}(x,y)$ is an absolutely irreducible factor of $g_t(x,y)$ in $\mathbb{F}_p[x,y]$. The reasoning is similar assuming that condition (2) holds.

\end{proof}

Assuming the irreducibility of $g_t(x,y)$ in $\mathbb{F}_p$ it is possible to relax the hypotheses involving the numbers $e_d$ (associated with the divisors  of $t$ and $t-1$) given in Theorem \ref{spiderman}. This is shown in the next lemma and theorem.

\begin{lemma}\label{newlemma}
Assume that $g_t(x,y)$ is irreducible over $\mathbb{F}_p$. Let $K$ be an extension of $\mathbb{F}_p$ such that there exists a non-singular point $P\in \chi_t(K)$. Then the absolutely irreducible factors of $g_t(x,y)$ have coefficients in $K$. 

\end{lemma}

\begin{proof}
Applying Lemma \ref{IfIrreducibleEqualDegreeFactors} we have that there exists a $K$-irreducible polynomial $H(X,Y,Z)$ such that the factorization of $G_t(X,Y,Z)$ into $K$-irreducible polynomials is
$$
G_t(X,Y,Z)=c\prod_{\sigma\in G}\sigma(H(X,Y,Z)),
$$
where $G=Gal(K/\mathbb{F}_{q})$ and $c\in \mathbb{F}_{p}$. We can assume, without loss of generality, that $H(P)=0$.

We claim that $H(X,Y,Z)$ is absolutely irreducible. Indeed, reasoning by contradiction we have that, if we assume that $H(X,Y,Z)$ is not absolutely irreducible and applying again Lemma \ref{IfIrreducibleEqualDegreeFactors} to $H$, it holds that $H(X,Y,Z)$ factorizes into $n\geq 2$ absolutely irreducible polynomials which are conjugate over $Gal(K'/K)$ for some extension $K'$ of $K$. Since $P$ is a $K$-rational point, this implies that any of these $n$ factors vanishes at $P$ and, therefore, the multiplicity of $H$ at $P$ must be greater than 2. This is a contradiction because $P$ is a non-singular point of $\chi_t$.

Therefore the above given factorization of $G_t(X,Y,Z)$ is, in fact, its factorization into absolutely irreducible polynomials. The result follows taking affine coordinates.

\end{proof}

\begin{theo}
Assume that $g_t(x,y)$ is irreducible over $\mathbb{F}_p$ and consider the set
$$E:=\{d\in \mathbb{N}\mid d>2 \mbox{ and} \mbox{ $d$ divides either $t$ or $t-1$}\}.$$
If $gcd(e_d\mid d\in E)=1$ then $g_t(x,y)$ is absolutely irreducible.

\end{theo}

\begin{proof}
Let $d\in E$ and assume first that $d$ is a divisor of $t$. Let $\eta$ be a primitive $d$-root of unity. Since $d>2$ one has that $\eta\not\not\in \{1,-1\}$ and, then, it is clear that the point $(-2^{-1}:\frac{1}{\eta-1}:1)\in \mathbb{P}^2$ is a non-singular point of $\chi_t(\mathbb{F}_{e_d})$.

Assume now that $d$ divides $t-1$ and let $\delta$ be a primitive $d$-root of unity. Since $d>2$ we have that $\delta\not=1$ and, then, $(\delta:1:0)$ is a non-singular point of $\chi_t(\mathbb{F}_{e_d})$.

By Lemma \ref{newlemma}, the absolutely irreducible factors of $g_t(x,y)$ have coefficients in the intersection of all the fields $\mathbb{F}_{e_d}$ for $d\in E$, that is $\mathbb{F}_p$ taking into account our assumptions.

\end{proof}





\begin{thebibliography}{99}
 
\bibitem{Beauville}
A.~Beauville,
{\it Complex Algebraic Surfaces}, London Math. Society, Student Texts 34,
Cambridge University Press, 1996.

\bibitem{Bodin} A.~Bodin, Reducibility of rational functions in several variables, 
{\em Israel J. Math.} {\bf 164} (2008), 333--347. 



\bibitem{cgm} A.~Campillo; G.~Gonzalez-Sprinberg; F.~Monserrat, Configurations of infinitely near points, \emph{Sao Paulo J. Math. Sci.} {\bf 3} (2009), no. 1, 115--160. 

\bibitem{Coulter-Matthews} Coulter, R; Matthews, R; 
{\em Planar Functions and Planes of Lenz-Barlotti Class II}
Designs Codes and Cryptography, 10 (1997) 167--184.

\bibitem{C}
Coulter, R., private communication.

\bibitem{DO}
P. Dembowski and T.G. Ostrom, Planes of order $n$ with collineation groups of
order $n^2$,
Math. Z. 103 (1968), 239--258.

\bibitem{Eisenbud}
D.~Eisenbud; J.~Harris,
{\it The Geometry of Schemes}, Grad. Texts Math., vol. 197,
Springer, New York-Heidelberg-Berlin, 1999.





\bibitem{F}
W.~Fulton,
{\it Algebraic\ Curves}, Benjamin, New York, 1969.

\bibitem{G1}
D.~G.~Glynn,
``Two New Sequences of Ovals in Finite Desarguesian Planes
of Even Order,'' in {\it Combinatorial Mathematics} X,
Springer Lecture Notes in Mathematics {\bf 1036},
Springer-Verlag, 1983, 217--229.

\bibitem{G2}
D.~G.~Glynn,
``A Condition for the Existence of Ovals in $PG(2,q)$, $q$ even,''
{\it Geometriae Dedicata} {\bf 32} (1989), 247--252.

\bibitem{GRZ}
R. Guralnick, J. Rosenberg and M. Zieve, 
A new family of exceptional polynomials in characteristic two, 
{\it Annals of Math.} 172 (2010), 1367Ð1396.

\bibitem{HM1}
F. Hernando, G. McGuire,
Proof of a conjecture on the sequence of exceptional numbers, classifying cyclic codes and APN functions. 
{\it J. Algebra} {\bf 343} (2011), 78Ð92.

\bibitem{HM2}
F. Hernando, G. McGuire,
Proof of a conjecture of Segre and Bartocci on monomial hyperovals in projective planes.
{\it Designs, Codes and Cryptography}
{\bf 65} (2012), Issue 3, pp 275-289.

\bibitem{H}
J.~W.~P.~Hirschfeld,
{\it Projective Geometries over Finite Fields},
Clarendon Press, Oxford, 1989.

\bibitem{H2}
J.~W.~P.~Hirschfeld,
``Ovals in Desarguesian Planes of Even Order,''
{\it Ann. Mat. Pura Appl.} {\bf 102} (1975), 79--89.

 \bibitem{Janwa-McGuire-Wilson} Janwa, H.; McGuire, G.; Wilson, R. M.; {\em Double-error-correcting cyclic codes and absolutely irreducible polynomials over ${\rm GF}(2)$.}; J. Algebra 178 (1995), no. 2, 665--676.

\bibitem{Hartshorne}
R.~Hartshorne,
{\it Algebraic Geometry}, Grad. Texts Math., vol.52,
Springer, New York-Heidelberg-Berlin, 1977.


\bibitem{Iitaka} S.~Iitaka,
{\em Algebraic Geometry}. Grad. Texts Math., vol.76,
Springer, New York-Heidelberg-Berlin, 1982.



\bibitem{kleiman} S.L.~Kleiman,
{\em Bertini and his two fundamental theorems}. Studies in the history of modern mathematics, \emph{III. Rend. Circ. Mat. Palermo} (2) Suppl. No. 55 (1998), 9--37. .

\bibitem{Kopparty-Yekhanin} S. Kopparty; S. Yekhanin; {\em Detecting Rational Points on Hypersurfaces over Finite Fields}; Computational Complexity, 2008. CCC '08. 23rd Annual IEEE Conference 2008, 311 - 320.

\bibitem{EL} Leducq, E.,
Functions which are PN on infinitely many extensions of $F_p$, $p$ odd,
\text{http://arxiv.org/abs/1006.2610}


\bibitem{LN}
R.~Lidl and H.~Niederreiter,
{\it Finite Fields}, Encyclopedia Math. Appl., vol.20,
Addison-Wesley, Reading, MA, 1983.

\bibitem{LN1}
R.~Lidl and H.~Niederreiter,
{\it Finite Fields}, Introduction to finite fields and their applications. 
Revision of the 1986 first edition. Cambridge University Press, Cambridge, 1994. xii+416 pp.

\bibitem{M}
R.~Matthews,
``Permutation Properties of the Polynomials $1+x+\cdots +x^k$
over a Finite Field,'' {\it Proc. A.M.S.}, {\bf 120} (1994), 47--51.

\bibitem{NK}
K. Nyberg and L.R.Knudsen, {\em Provable security against differential cryptanalysis}, Advances in
Cryptology  CRYPTO92, LNCS 740, pp. 566-574,Springer-Verlag, 1992.


 
\bibitem{SC}
W.~M.~Schmidt,
{\it Equations over Finite Fields}, Springer Lecture Notes
in Mathematics, {\bf 536}, Springer-Verlag, 1976, 210.

\end{thebibliography}
\end{document}